\documentclass[letterpaper,11pt]{amsart}
\usepackage{indentfirst} 
\usepackage{amssymb}
\usepackage{mathrsfs}
\usepackage{amsmath}
\usepackage{amsthm}
\usepackage{thmtools}
\usepackage{enumitem}
\usepackage{bbm}            
\usepackage[colorlinks = true, linkcolor = black, citecolor = blue ]{hyperref}
\usepackage[usenames,dvipsnames]{xcolor} 
\usepackage{tikz}
\usepackage[left=3.2cm, right=3.2cm, bottom=3.4cm]{geometry}

\theoremstyle{plain}
\newtheorem{theorem}{Theorem}[section]

\theoremstyle{plain}
\newtheorem{proposition}{Proposition}[subsection]
\newtheorem{lemma}[proposition]{Lemma}
\newtheorem{corollary}[proposition]{Corollary}

\theoremstyle{definition}
\newtheorem{definition}{Definition}[subsection]

\theoremstyle{remark}
\newtheorem{remark}{Remark}
\newtheorem{claim}{Claim}
\declaretheorem[name=Acknowledgements,numbered=no]{ack}

\newcommand{\vertiii}[1]{{\left\vert\kern-0.25ex\left\vert\kern-0.25ex\left\vert #1 
    \right\vert\kern-0.25ex\right\vert\kern-0.25ex\right\vert}}

\def\e{\epsilon}
\def\R{{\mathbb R}}
\def\N{{\mathbb N}}

\def\M{{\mathcal M}}
\def\E{{\mathcal E}}
\def\F{{\mathcal F}}
\def\G{{\mathcal G}}
\def\L{{\mathscr L}}
\def\vphi{{\varphi}}
\def\T{{\bold T}}

\begin{document}

\title[The Vlasov--Poisson system with a trapping potential]{Small data solutions for the Vlasov--Poisson system with a trapping potential}

\author[Anibal Velozo Ruiz]{Anibal Velozo Ruiz} \address{Facultad de matem\'atica, Pontificia Universidad Cat\'olica de Chile, Avenida Vicu\~na Mackenna 4860, Santiago, Chile.}
\email{apvelozo@mat.uc.cl}

\author[Renato Velozo Ruiz]{Renato Velozo Ruiz} \address{Laboratory Jacques-Louis Lions (LJLL), University Pierre and Marie Curie (Paris 6), 4
place Jussieu, 75252 Paris, France.}
\email{ravelozor@gmail.com}

\begin{abstract}
In this paper, we study small data solutions for the Vlasov--Poisson system with the simplest external potential, for which unstable trapping holds for the associated Hamiltonian flow. We prove sharp decay estimates in space and time for small data solutions to the Vlasov--Poisson system with the unstable trapping potential $\frac{-|x|^2}{2}$ in dimension two or higher. The proofs are obtained through a commuting vector field approach. We exploit the uniform hyperbolicity of the Hamiltonian flow, by making use of the commuting vector fields contained in the stable and unstable invariant distributions of phase space for the linearized system. In dimension two, we make use of modified vector field techniques due to the slow decay estimates in time. Moreover, we show an explicit teleological construction of the trapped set in terms of the non-linear evolution of the force field.
\end{abstract}

\maketitle

\setcounter{tocdepth}{1}
\tableofcontents

\section{Introduction}\label{introduction_small_data_unstable_potential}
In this paper, we study the evolution in time of collisionless many-particle systems on $\R^n$, which are described statistically by a distribution function on phase space that satisfies a non-linear PDE system motivated by kinetic theory. More precisely, we investigate the non-linear dynamics of solutions $f(t,x,v)$ to the \emph{Vlasov--Poisson system with an external potential $\Phi(x)$}; given by 
\begin{equation}\label{vlasov_poisson_external_potential}
\begin{cases}
\partial_t f+v\cdot\nabla_xf-(\nabla_x\Phi+\mu\nabla_x\phi) \cdot \nabla_vf=0,\\ 
\Delta_x \phi=\rho(f),\\
\rho(f)(t,x):=\int_{\R^d}f(t,x,v)dv,\\
f(t=0,x,v)=f_0(x,v),
\end{cases}
\end{equation}
where $t\in [0,\infty)$, $x\in \R^n_x$, $v\in \R^n_v$, and $\mu\in \{1,-1\}$ is a fixed constant. According to the value of $\mu$, the interaction between the particles of the system is either \emph{attractive} (when $\mu=1$), or \emph{repulsive} (when $\mu=-1$). The nonlinearity of this classical kinetic PDE system arises from the mean field generated by the many-particle system, through the gradient of the solution to the Poisson equation, which is determined in terms of the so-called \emph{spatial density} $\rho(f)$, defined by integrating the distribution function in the velocity variables. 

The Vlasov--Poisson system with an external potential $\Phi$, describes a collisionless many-particle system for which the trajectories described by its particles are set by the mean field generated by the many-particle system, and an external potential $\Phi$ motivated by specific considerations of the problem at hand. External potentials have been previously used in the literature to study collisional and collisionless many-particle systems in kinetic theory \cite{HN04, He07, DMS09,D11, DL12, DMS15, CDHMM21, CL21}. The Vlasov--Poisson system with an external potential is motivated by the classical \emph{Vlasov--Poisson system}, given precisely by the Vlasov--Poisson system with a vanishing external potential. 

The Vlasov--Poisson system was originally introduced for the study of galactic dynamics by Jeans \cite{Je15}, when the interaction between the particles of the system is attractive ($\mu=1$). In this setup, the field $\nabla_x\phi$ is also known as the \emph{gravitational field}. Independently, the Vlasov--Poisson system was introduced for the study of plasma physics by Vlasov \cite{Vla68}, when the interaction between the particles of the system is repulsive ($\mu=-1$). In this setup, the field $\nabla_x\phi$ is also known as the \emph{electric field}. We note that in the plasma physics case, the many-particle system (\ref{vlasov_poisson_external_potential}) is composed by a single species of particles without global neutrality. The field $\nabla_x\phi$ for the Vlasov--Poisson system with an external potential has the same meaning in both the attractive and the repulsive case. Subsequently, the Vlasov--Poisson system has been widely used to research collisionless many-particle systems in astrophysics \cite{BT11} and plasma physics \cite{LP81}.

The Vlasov--Poisson system is a non-linear transport--elliptic type PDE system whose rich dynamics have been extensively studied in the scientific literature. The first well-posedness result for this PDE system was obtained by Okabe and Ukai \cite{OU78} who proved global well-posedness in dimension two and local well-posedness in dimension three. Later in time, a large class of non-trivial stationary solutions for this system were constructed  \cite{BFH86, RR00, BT11}. Seminal independent works by Pfaffelmoser \cite{Pf92} and Lions--Perthame \cite{LP91} of the early nineties proved \emph{global well-posedness} for the Vlasov--Poisson system in dimension three (see also Schaeffer's proof \cite{Sch91}). These global well-posedness results can be adapted to incorporate an external potential $\Phi(x)$, as long as $\nabla_x\Phi$ has Lipschitz regularity (see the introduction of \cite{GHK12}). However, the description of the non-linear dynamics of solutions to the Vlasov--Poisson system for arbitrary finite energy data is not yet fully understood. Nonetheless, non-linear perturbative stability results for stationary solutions of this PDE system have been proved. Orbital stability under spherically symmetric perturbations has been proved for several non-increasing spherically symmetric stationary solutions \cite{Sch04,SS06,GR07, GL08,LMR08,LMR11}. We stress the work by Lemou, M\'ehats, and Rapha\"{e}l \cite{LMR12}, who proved orbital stability under \emph{arbitrary perturbations} for a large class of non-increasing spherically symmetric stationary solutions previously considered in the literature. We also comment on the asymptotic stability of a point charge for the repulsive Vlasov--Poisson system in dimension three by Pausader, Widmayer, and Yang \cite{PWY22}. 

The first asymptotic stability result for solutions to the Vlasov--Poisson system was obtained by Bardos and Degond \cite{BD85}, who studied the evolution in time of small data solutions for the Vlasov--Poisson system for compactly supported initial data, using the method of characteristics. Later on, this small data global existence result for the Vlasov--Poisson system was improved by Hwang, Rendall and Vel\'asquez \cite{HRV11} who proved optimal time (but not spatial) decay estimates for higher order derivatives of the spatial density for compactly supported data, again using the method of characteristics. More recently, the stability of the vacuum solution for the Vlasov--Poisson system à la Bardos--Degond was revisited by Smulevici \cite{Sm16}, who proved stability based upon energy estimates using a vector field method. As a result, Smulevici \cite{Sm16} obtained propagation in time of a global energy bound, in terms of commuted Vlasov fields associated with conservation laws of the free transport operator, and optimal space and time decay estimates for the spatial density induced by the distribution function. Later Duan \cite{Du22} simplified the functional framework used to prove the stability of the vacuum solution for the Vlasov--Poisson system in \cite{Sm16}. See \cite{IPWW} for another proof of the stability of vacuum using methods coming from dispersive PDEs.

In this paper, we are interested in stability results for \emph{dispersive} collisionless many-particle systems for which the dynamics described by the particles of the system are \emph{hyperbolic}. Motivated by this class of many-particle systems, we consider the Vlasov--Poisson system with the simplest external potential, for which \emph{unstable trapping} is expected to hold for the Hamiltonian flow associated to small data solutions of this system. For the purposes of this paper, we say that \emph{unstable trapping} holds for a Hamiltonian flow in $\R^n_x\times \R^n_v$ if the trajectories of the flow escape to infinity for every point in phase space, except for a non-trivial set of measure zero for which the future of every trajectory of the flow is always bounded. More precisely, we study the non-linear dynamics of small data solutions for the Vlasov--Poisson system with the external potential $\frac{-|x|^2}{2}$. We note that unstable trapping holds trivially for the linear Vlasov equation with the external potential $\frac{-|x|^2}{2}$. As a result, we prove asymptotic stability for small data solutions to the Vlasov--Poisson systen with the external potential $\frac{-|x|^2}{2}$ in dimension higher or equal to two, by using a commuting vector field method à la Smulevici. 

We investigate this toy model in order to offer insights on the study of stability results for dispersive collisionless many-particle systems for which the associated Hamiltonian flow is hyperbolic. This dispersive behaviour holds locally for 1D Hamiltonian flows arising from potentials with a global maximum in a neighborhood of the associated fixed hyperbolic point. An important example of dispersive collisionless many-particle systems, for which the Hamiltonian flow is hyperbolic, is given by many-particle systems in the exterior of black hole backgrounds which admit a \emph{normally hyperbolic trapped set} \cite{WZ11, D15}.\footnote{We stress the trapped set in the exterior of black hole backgrounds is \emph{eventually absolutely $r$-normally hyperbolic} for every $r$ according to \cite[Chapter 1, Definition 4]{HPS77}.}

\subsection{The main results}
In this manuscript, we investigate the non-linear dynamics of small data solutions for the Vlasov--Poisson system with the external potential $\frac{-|x|^2}{2}$; given by
\begin{equation}\label{vlasov_poisson_unstable_trapping_potential_paper}
\begin{cases}
\partial_t f+v\cdot\nabla_xf+x\cdot \nabla_vf-\mu\nabla_x\phi \cdot \nabla_vf=0,\\
\Delta_x \phi=\rho(f),\\
\rho(f)(t,x):=\int_{\R^d}f(t,x,v)dv,\\
f(t=0,x,v)=f_0(x,v),
\end{cases}
\end{equation}
where $t\in [0,\infty)$, $x\in \R^n_x$, $v\in \R^n_v$, and $\mu\in\{1,-1\}$ is a fixed constant. The local well-posedness theory for this PDE system is standard (see for instance \cite[Section 3]{HK19}). In dimension greater than two, we study the evolution in time of small initial distribution functions $f_0:\R^{n}_x\times \R^{n}_v\to [0,\infty)$, in the energy space defined by a higher order Sobolev norm: $$\mathcal{E}_N[f]:=\sum_{|\alpha|\leq N}\sum_{ Z^{\alpha}\in {\lambda}^{|\alpha|}}\|Z^{\alpha}f\|_{L^1_{x,v}},$$ where $Z^{\alpha}$ are differential operators of order $|\alpha|$, obtained as compositions of vector fields, in a class ${\lambda}$ of commuting vector fields for the linear Vlasov equation with the trapping potential $\frac{-|x|^2}{2}$. This linear Vlasov equation corresponds to the linearization of the Vlasov--Poisson system with the external potential $\frac{-|x|^2}{2}$, with respect to its vacuum solution. See Subsection \ref{subsection_macro_micro_vector_field} for the precise definition of $\lambda$.

\begin{theorem}\label{theorem_stability_vacuum_external_potential}
Let $n\geq 3$ and $N\geq 2n$. There exists $\epsilon_0>0$ such that for all $\epsilon\in (0,\epsilon_0)$, if the initial data $f_0$ for the Vlasov--Poisson system with the trapping potential $\frac{-|x|^2}{2}$ on $\R^n_x\times\R_v^n$ satisfies $\mathcal{E}_N[f_0]\leq \epsilon$. Then, the corresponding solution $f$ for the Vlasov--Poisson system with the trapping potential $\frac{-|x|^2}{2}$ exists globally, and it satisfies the following estimates for every $t\in [0,\infty)$ and $x\in \R^n$:
\begin{enumerate}[label = (\roman*)]
    \item Global energy estimate: 
    $$\mathcal{E}_N[f(t)]\leq 2\epsilon.$$
    \item Decay in space and time of the spatial density for any multi-index $\alpha$ of order $|\alpha|\leq N-n$:
    $$|\rho(Z^{\alpha}f)(t,x)|\leq \dfrac{C_{N,n}\epsilon}{(e^t+|x|)^{n}},$$
    as well as improved decay estimates for its derivatives $$|\partial_x^{\alpha}\rho(f)(t,x)|\leq \dfrac{C_{N,n}\epsilon}{(e^t+|x|)^{n+|\alpha|}},$$ where $C_{N,n}>0$ is a uniform constant depending only on $n$ and $N$.
\end{enumerate}
\end{theorem}

In the two-dimensional case, we study the evolution in time of small initial distribution functions $f_0:\R^{n}_x\times \R^{n}_v\to [0,\infty)$, in the energy space defined by a higher order Sobolev norm: $$\mathcal{E}^m_N[f]:=\sum_{|\alpha|\leq N}\sum_{ Y^{\alpha}\in {\lambda_m}^{|\alpha|}}\|Y^{\alpha}f\|_{L^1_{x,v}},$$ where $Y^{\alpha}$ are differential operators of order $|\alpha|$, obtained as compositions of modified vector fields in a class ${\lambda}_m$. The vector fields in ${\lambda}_m$ are modifications of the commuting vector fields for the linear Vlasov equation with the trapping potential $\frac{-|x|^2}{2}$ in $\lambda$. See Subsection \ref{subsection_modified_vector_fields} for the precise definition of $\lambda_m$.

\begin{theorem}\label{theorem_stability_vacuum_external_potential_2D}
Let $N\geq 7$. There exists $\epsilon_0>0$ such that for all $\epsilon\in (0,\epsilon_0)$, if the initial data $f_0$ for the Vlasov--Poisson system with the trapping potential $\frac{-|x|^2}{2}$ on $\R^2_x\times\R_v^2$ satisfies $\mathcal{E}_N[f_0]\leq \epsilon$. Then, the corresponding solution $f$ for the two dimensional Vlasov--Poisson system with the trapping potential $\frac{-|x|^2}{2}$ exists globally, and it satisfies the following estimates for every $t\in [0,\infty)$ and $x\in \R^2$:
\begin{enumerate}[label = (\roman*)]
    \item Global energy estimate: 
    $$\mathcal{E}^m_N[f(t)]\leq 2\epsilon.$$
    \item Decay in space and time of the spatial density for any multi-index $\alpha$ of order $|\alpha|\leq N-2$:
    $$|\rho(Z^{\alpha}f)(t,x)|\leq \dfrac{C_{N}\epsilon}{(e^t+|x|)^{2}},$$
    as well as improved decay estimates for its derivatives $$|\partial_x^{\alpha}\rho(f)(t,x)|\leq \dfrac{C_{N}\epsilon}{(e^t+|x|)^{2+|\alpha|}},$$ where $C_{N}>0$ is a uniform constant depending only on $N$.
\end{enumerate}
\end{theorem}

\begin{remark}
\begin{enumerate}[label = (\roman*)]
    \item The proofs of Theorem \ref{theorem_stability_vacuum_external_potential} and Theorem \ref{theorem_stability_vacuum_external_potential_2D} fit into the general framework of the vector field method for dispersive collisionless kinetic equations developed in \cite{Sm16}, using weighted Sobolev estimates in terms of commuting vector fields. As a result, we obtain sharp decay estimates in space and time for the induced spatial density by exploiting the weights of the corresponding commuting vector fields.
    \item We exploit the uniform hyperbolicity of the non-linear Hamiltonian flow, by making use of the commuting vector fields contained in the stable and unstable invariant distributions of phase space\footnote{We refer to a \emph{distribution} in phase space $\R^n_x\times\R^n_v$ as a map $(x,v)\mapsto \Delta_{(x,v)}\subseteq T_{(x,v)}(\R^n_x\times\R^n_v)$, where $\Delta_{(x,v)}$ are vector subspaces satisfying suitable conditions (in the standard sense used in differential geometry).} for the linearized system. In dimension two, we make use of modified vector field techniques due to the slow decay estimates in time. The modifications to the commuting vector fields for the linearized system \emph{grow linearly in time}. This is in contrast with previous applications of modified vector fields to collisionless kinetic equations where the modifications \emph{grow logarithmically in time}. See for instance \cite{Sm16, B20, FJS21}. As a result, we obtain exponential decay in time for the induced spatial density. The rate of exponential decay for the spatial density coincides with the \emph{sum of all positive Lyapunov exponents} of the Hamiltonian flow.
\item The decay assumed in the velocity variable of the initial distribution functions in Theorem \ref{theorem_stability_vacuum_external_potential} and Theorem \ref{theorem_stability_vacuum_external_potential_2D} is \emph{optimal}. The integrability in the velocity variable of the distribution function is required to make sense of the Poisson equation classically. Similar assumptions are made for derivatives of $f$. In particular, Theorem \ref{theorem_stability_vacuum_external_potential} and Theorem \ref{theorem_stability_vacuum_external_potential_2D} allow initial distribution functions with infinite total Hamiltonian energy.
\end{enumerate}
\end{remark}

Let $f$ be a small data solution of \eqref{vlasov_poisson_unstable_trapping_potential_paper}, according to the assumptions in Theorem \ref{theorem_stability_vacuum_external_potential} or Theorem \ref{theorem_stability_vacuum_external_potential_2D}. The particle dynamics along which the distribution function $f$ is transported corresponds to the \emph{characteristic flow} given by 
\begin{equation}\label{characteristics_NL_system_section_trap_intro}
\frac{d}{dt}X(t,x,v)=V(t,x,v),\qquad \frac{d}{dt}V(t,x,v)=X(t,x,v)-\mu\nabla_x\phi (t,X(t,x,v)),
\end{equation}
with the initial data $X(0,x,v)=x$ and $V(0,x,v)=v$. The characteristics are well-defined by the classical Cauchy--Lipschitz theorem. In the proofs of Theorem \ref{theorem_stability_vacuum_external_potential} and Theorem \ref{theorem_stability_vacuum_external_potential_2D}, we show that $\nabla_x\phi$ decays exponentially in time. Thus, the characteristic flow \eqref{characteristics_NL_system_section_trap_intro} determines a decaying perturbation of the linearized particle system as $t\to\infty$. For this reason, one expects that unstable trapping holds for the characteristic flow \eqref{characteristics_NL_system_section_trap_intro}.

\begin{definition}
Let $(X(t,x,v),V(t,x,v))$ be a solution of the characteristic flow \eqref{characteristics_NL_system_section_trap_intro}. We say that $(x,v)$ \emph{escapes to infinity}, if $\|(X(t),V(t))\|\to \infty$ as $t\to\infty$. If $(x,v)$ does not escape to infinity, we call $(x,v)$ \emph{trapped}. We denote by $\Gamma_+\subset \R^n_x\times\R^n_v$ the union of all trapped $(x,v)$. We call $\Gamma_+$ the \emph{trapped set}.
\end{definition}

Since the force field $\nabla_x\phi$ decays exponentially in time, the origin $\{x=0,v=0\}$ is \emph{formally} a fixed point of \eqref{characteristics_NL_system_section_trap_intro} when $t\to \infty$. Applying the stable manifold theorem \cite[Theorem 2.6]{Hin21} for decaying perturbations of time-translation-invariant dynamical systems with hyperbolic trapping, the set $$W^s(0,0):=\Big\{(x,v)\in \R^n_x\times\R^n_v: (X(t,x,v),V(t,x,v))\to (0,0) \text{   as   } t\to \infty\Big\}$$ defines the stable manifold of the origin. Furthermore, $W^s(0,0)$ is an $n$-dimensional invariant manifold of class $C^1$ which converges to the trapped set of the linearized system $\{x+v=0\}$ when $t\to \infty$.

In the specific case of a decaying perturbation \eqref{characteristics_NL_system_section_trap_intro} of the linearized particle system, we identify \emph{explicitly} the stable manifold $W^s(0,0)$ in terms of the non-linear evolution in time of the force field $\nabla_x\phi$. Moreover, we characterize the trapped set $\Gamma_+$ as the stable manifold $W^s(0,0)$.

\begin{theorem}\label{thm_characterization_trapped_set}
Let $f_0$ be an initial data for the Vlasov--Poisson system with the trapping potential $\frac{-|x|^2}{2}$ on $\R^n_x\times\R_v^n$, such that $\mathcal{E}_N[f_0]< \epsilon_0$. Let $\Gamma_+$ be the trapped set of the characteristic flow \eqref{characteristics_NL_system_section_trap_intro} associated to the corresponding solution $f$ of \eqref{vlasov_poisson_unstable_trapping_potential_paper}. Then, the trapped set $\Gamma_+$ is equal to the $n$-dimensional stable manifold of the origin $W^s(0,0)$ of class $C^{N-n-1}$. Moreover, the trapped set is characterized as
\begin{equation}\label{characterization_stable_manifold_intro}
\Gamma_+=\Big\{(x,v): x+v=\int_0^{\infty}\frac{1}{e^{t'}}\mu\nabla_x\phi(t',X(t',x,v))dt' \Big\},
\end{equation}
where we have $$\Big|\int_0^{\infty}\frac{1}{e^{t'}}\mu\nabla_x\phi(t',X(t',x,v))dt'\Big|\leq C_{N,n}\e_0,$$ with $C_{N,n}>0$ a uniform constant depending only on $n$ and $N$.
\end{theorem}

\begin{remark}
\begin{enumerate}[label = (\roman*)]
\item In the proof of Theorem \ref{thm_characterization_trapped_set}, we first characterize the set $W^s(0,0)$ as the right hand side of \eqref{characterization_stable_manifold_intro}, and later we show this set is a non-empty invariant manifold of class $C^1$. We find the characterization \eqref{characterization_stable_manifold_intro} of the trapped set by integrating in time the characteristic flow. In particular, we do \emph{not} apply the stable manifold theorem \cite[Theorem 2.6]{Hin21} to obtain Theorem \ref{thm_characterization_trapped_set}.
\item The characterization of the trapped set \eqref{characterization_stable_manifold_intro} gives an \emph{explicit teleological construction} of $\Gamma_+$ in terms of the non-linear evolution in time of the force field $\nabla_x\phi$. In particular, one can easily show that $\Gamma_+$ converges quantitatively to the trapped set of the linearized system $\{x+v=0\}$ when $t\to \infty$.
\end{enumerate}
\end{remark}

\subsubsection{Previous non-linear stability results for dispersive collisionless many-particle systems using vector fields methods}

Vector field methods have been developed to obtain \emph{robust} techniques to prove asymptotic stability results for stationary solutions of non-linear evolution equations. We stress the classical vector field method developed by Klainerman \cite{Kl85} for the study of the wave equation in Minkowski spacetime, which allows the proof of quantitative decay estimates in space and time for solutions to the wave equation, based on weighted Sobolev estimates, using energy norms in terms of commuting vector fields arising from the symmetries of spacetime. The vector field method has shown to be a powerful technique for the study of quasilinear systems of wave equations, such as the Einstein vacuum equations \cite{CK93, LR10}.

The study of vector field methods for dispersive collisionless many-particle systems was pioneered by Smulevici \cite{Sm16} who developed a vector field method for this class of kinetic systems, inspired by the classical vector field method for wave equations introduced in \cite{Kl85}. Specifically Smulevici \cite{Sm16} proved the stability of the vacuum solution for the Vlasov--Poisson system using this methodology. Later Duan \cite{Du22} simplified the functional framework used to prove the asymptotic stability result in \cite{Sm16}. Smulevici \cite{Sm16} was motivated by the work of Fajman, Joudioux, and Smulevici \cite{FJS17}, who developed a vector field method to prove decay estimates in space and time for the spatial density induced by solutions to the \emph{relativistic} Vlasov equation in Minkowski spacetime.\footnote{In general relativity, many-particle systems can be composed by particles moving at the speed of light for which their mass vanishes. Nonetheless, we only comment on stability results for relativistic collisionless many-particle systems for which the mass of their particles is one.} Furthermore, \cite{FJS17} made use of a vector field method to prove stability results for the vacuum solution of the Vlasov--Nordstr\"{o}m system. Later on, Fajman, Joudioux, and Smulevici \cite{FJS21} once again used a vector field method to study dispersive collisionless many-particle systems in a neighborhood of Minkowski spacetime, under the geometric framework of general relativity. In other words, the authors of \cite{FJS21} proved the stability of Minkowski spacetime as a solution of the Einstein--Vlasov system. We emphasize that Taylor and Lindblad \cite{LT20} independently proved the stability of Minkowski spacetime as a solution of the Einstein--Vlasov system by also using a vector field method. 

The vector field method for dispersive collisionless many-particle systems has also been used by Bigorgne \cite{B20, B21, B22}, in order to prove the stability of vacuum for the relativistic Vlasov--Maxwell in dimension greater or equal to three. Wang \cite{Wa22} obtained another proof of the stability of vacuum for the relativistic Vlasov--Maxwell in dimension three, by using a combination of the vector field method and Fourier techniques. We emphasize that the stability of the vacuum solution for the relativistic Vlasov--Maxwell system had been first shown by Glassey and Schaeffer \cite{GS87} using the method of characteristics.

\subsection{Outline of the paper}
The remainder of the paper is structured as follows.
\begin{itemize}
    \item \textbf{Section \ref{preliminaries_unstable_potential}.} We study the linearization of the non-linear Vlasov--Poisson system with the potential $\frac{-|x|^2}{2}$, with respect to its vacuum solution. We introduce the class of vector fields used to define the energy norm seen in Theorem \ref{theorem_stability_vacuum_external_potential}. We conclude with some basic lemmata for the commuted equations.
    \item \textbf{Section \ref{section_exponential_decay_velocity_averages}.} We prove weighted Sobolev inequalities for the induced spatial density of a distribution function by making use of commuting vector fields. We obtain decay in space and time of the spatial density induced by solutions to the linear Vlasov equation with the potential $\frac{-|x|^2}{2}$. 
    \item \textbf{Section \ref{section_proof_theorem_dimension_higher_two}.} We prove global existence of small data solutions for the Vlasov--Poisson system with the potential $\frac{-|x|^2}{2}$ in dimension greater than two.
    \item \textbf{Section \ref{section_proof_theorem_dimension_two}.} We prove global existence of small data solutions for the Vlasov--Poisson system with the potential $\frac{-|x|^2}{2}$ in dimension two using modified vector fields.
    \item \textbf{Section \ref{section_proof_trapped_set_characteristic_flow}.} We characterize the trapped set of the characteristic flow associated to the small data solutions studied in the previous sections. 
\end{itemize}

\begin{ack}
RVR would like to express his gratitude to his advisors Mihalis Dafermos and Cl\'ement Mouhot for their continued guidance and encouragements. RVR also would like to thank L\'eo Bigorgne and Jacques Smulevici for many helpful discussions. RVR received funding from the ANID grant 72190188, the Cambridge Trust grant 10469706, and the European Union’s Horizon 2020 research and innovation programme under the Marie Skłodowska-Curie grant 101034255. AVR received funding from the grant FONDECYT Iniciaci\'on 11220409.
\end{ack}

\section{Preliminaries}\label{preliminaries_unstable_potential}

In this section, we introduce the set of commuting vector fields used to study dispersion for the non-linear Vlasov--Poisson system with the external potential $\frac{-|x|^2}{2}$ building upon the dynamics defined by the flow map associated to the characteristics of the linear Vlasov equation with the same external potential. Furthermore, we prove useful lemmata which are going to be applied in the following section to show weighted Sobolev inequalities for the induced spatial density of a distribution function.

In the rest of the paper, the notation $A\lesssim B$ is repetitively used to specify that there exists a universal constant $C > 0$ such that $A \leq CB$, where $C$ depends only on the dimension $n$, the corresponding order of Sobolev regularity, or other fixed constants.

\subsection{The Vlasov equation with the potential \texorpdfstring{$\frac{-|x|^2}{2}$}{x2}}\label{subsection_vlasov_linear_flow}

In this subsection, we study the dynamics of the linearization of the non-linear Vlasov--Poisson system with the trapping potential $\frac{-|x|^2}{2}$ with respect to its vacuum solution, which is given by the \emph{linear Vlasov equation with the trapping potential $\frac{-|x|^2}{2}$} taking the form
\begin{equation}\label{vlasov_linear_flow}
\begin{cases}
\partial_t f+v\cdot\nabla_xf+x\cdot \nabla_vf=0,\\ 
f(t=0,x,v)=f_0(x,v),
\end{cases}
\end{equation}
where $f_0:\R^{n}_x\times \R^{n}_v\to [0,\infty)$ is a sufficiently regular initial data. We emphasize that this linear Vlasov equation is a transport equation along the Hamiltonian flow given by
\begin{equation}\label{linear_hyperbolic_ode_system}
    \dfrac{dx^i}{dt}=v^i,\qquad \dfrac{dv^i}{dt}=x^i,
\end{equation}
defined by the Hamiltonian system $(\R^{n}_x\times \R^{n}_v, H)$ in terms of the Hamiltonian  $$H(x,v):=\dfrac{1}{2}\sum_{i=1}^n(v^i)^2-\dfrac{1}{2}\sum_{i=1}^n(x^i)^2.$$ The Hamiltonian system $(\R^{n}_x\times \R^{n}_v,H)$ is \emph{completely integrable in the sense of Liouville} due to the $n$ independent conserved quantities in involution $$H^i(x,v):=\dfrac{1}{2}(v^i)^2-\dfrac{1}{2}(x^i)^2,$$ where $i\in \{1,2,\dots,n\}$, whose sum yields the total Hamiltonian $H$. In particular, we can write an explicit solution for the linear Vlasov equation (\ref{vlasov_linear_flow}) by computing the flow map precisely.

\begin{lemma}
Let $f_0$ be an initial data for the Vlasov equation (\ref{vlasov_linear_flow}). Then, the corresponding solution $f$ to the Vlasov equation (\ref{vlasov_linear_flow}) is given by
\begin{equation}
    f(t,x,v)=f_0\Big(x\cosh t-v\sinh t , v \cosh t-x \sinh t\Big).
\end{equation}
\end{lemma}

\begin{proof}
Integrating directly the Hamiltonian flow (\ref{linear_hyperbolic_ode_system}) satisfied by the characteristics of the linear Vlasov equation, we obtain $$(X_{\L}^i+V_{\L}^i)(t)=e^{t}(X_{\L}^i+V_{\L}^i)(0),\qquad (X_{\L}^i-V_{\L}^i)(t)=e^{-t}(X_{\L}^i-V_{\L}^i)(0),$$ for every $i\in \{1,2,\dots,n\}$. As a result, the flow map $\phi_t:\R^{n}_x\times \R^{n}_v\to \R^{n}_x\times \R^{n}_v$ defined by the characteristics of the Vlasov equation (\ref{vlasov_linear_flow}) is given by
\begin{equation}
    \phi_t(x,v):=(X_{\L}(t),V_{\L}(t))=\Big(x\cosh t +v\sinh t ,x\sinh t+v\cosh t \Big),
\end{equation}
which allows to write the solution of the linear Vlasov equation (\ref{vlasov_linear_flow}) by $$f(t,x,v)=f_0(\phi_{-t}(x,v))=(X_{\L}(-t),V_{\L}(-t))=f_0\Big(x\cosh t-v\sinh t , v \cosh t-x \sinh t\Big),$$ in terms of the initial distribution function $f_0$. 
\end{proof}

\subsection{Macroscopic and microscopic vector fields}\label{subsection_macro_micro_vector_field}

In this subsection, we introduce classes of vector fields contained in the tangent space of phase space used to study the dispersion of small data solutions for the non-linear Vlasov--Poisson system with the trapping potential $\frac{-|x|^2}{2}$ motivated by the explicit dynamics of the linear Vlasov equation (\ref{vlasov_linear_flow}). For this purpose, we introduce the following terminology: we say that a vector field is \emph{macroscopic} if it is contained in the tangent space of $\R^n_x$, and we say that a vector field is \emph{microscopic} if it is contained in the tangent space of $\R^n_x\times \R^n_v$.

Let us consider the generator of the Hamiltonian flow defined by the characteristics of the linear Vlasov equation (\ref{vlasov_linear_flow}) given by 
\begin{equation}
    X:=v \cdot \nabla_{x}+x\cdot \nabla_{v},
\end{equation}
and observe the linear Vlasov equation (\ref{vlasov_linear_flow}) can be written as
$$(\partial_t+X)f =0.$$

The commutators between the vector fields $\partial_{x^i}$, $\partial_{v^i}$ and $X$ are given by
\begin{equation*}
    [\partial_{x^i}, X ]=\partial_{v^i}, \quad 
    [\partial_{v^i},X]=\partial_{x^i},\quad [\partial_{x^i},\partial_{v^i}]=0,\quad 
    \text{for every $i\in \{1,2,\dots,n\}$.}
\end{equation*}
This allows us to exhibit several vector fields that commute with equation (\ref{vlasov_linear_flow}). More precisely, let us consider the following commuting microscopic vector fields
\begin{enumerate}[label = (\roman*)]
    \item unstable vector fields $U_i:=e^t(\partial_{x^i}+\partial_{v^i})$,
    \item stable vector fields $S_{i}:=e^{-t}(\partial_{x^i}-\partial_{v^i})$,
    \item scaling in phase space $L:=\sum_{i=1}^n x^i\partial_{x^i}+v^i\partial_{v^i}$,
    \item rotations $R_{ij}:=x^i\partial_{x^j}-x^j\partial_{x^i}+v^i\partial_{v^j}-v^j\partial_{v^i}$,
\end{enumerate}
and define 
\begin{equation*}
\lambda:=\Big\{ U_i, S_i, L, R_{ij}  \Big\},\quad  \lambda_0:=\Big\{ U_i, L, R_{ij}  \Big\},
\end{equation*}
where $i,j\in \{1,2,\dots,n\}$. The collection of microscopic vector fields $\lambda$ is used to set the energy space on which the distribution functions in this paper are defined.

\begin{lemma}\label{lemma_commutators_Vlasov_external_potential}
Let $f$ be a regular solution of the Vlasov equation with the trapping potential $\frac{-|x|^2}{2}$. Then, $Zf$ is also a solution of this equation for every $Z\in \lambda$. 
\end{lemma}
\begin{proof}
Observe that $[\partial_t+X,Z]=0$, for every $Z\in \lambda$. Thus, we have
$$(\partial_t+X)(Zf)=Z(\partial_t+X)f+[\partial_t+X,Z]f=0,$$ since $f$ is a solution of the linear Vlasov equation. Therefore, $Zf$ is a solution as well. 
\end{proof}

Observe that for every sufficiently regular solution $f$ to the linear Vlasov equation, the norm $\|f(t)\|_{L^1_{x,v}}$ is constant in time. In particular, we have that $$\|f(t)\|_{L^1_{x,v}}=\|f(0)\|_{L^1_{x,v}},$$ for every $t\ge 0$. A similar conservation law for derivatives of the distribution function follows from Lemma \ref{lemma_commutators_Vlasov_external_potential}. 

\begin{corollary}\label{corollary_conservation_law_vector_fields}
Let $f_0$ be a sufficiently regular initial data for the Vlasov equation (\ref{vlasov_linear_flow}). Then, the corresponding solution $f$ to the Vlasov equation (\ref{vlasov_linear_flow}) satisfies that $$\|Zf(t)\|_{L^1_{x,v}}=\|Z f(0)\|_{L^1_{x,v}},$$ for every $t\geq 0$, and every vector field $Z\in \lambda$.
\end{corollary}

\begin{remark}
In Section \ref{section_exponential_decay_velocity_averages}, we prove optimal space and time decay estimates for the spatial density induced by solutions to the linear Vlasov equation (\ref{vlasov_linear_flow}), even though spatial derivatives of the distribution function grow exponentially in time. This follows by using the commuting vector fields of the Vlasov equation contained in the invariant distributions of phase space, since the spatial derivatives of the distribution function can be written as
\begin{align*}
\partial_{x^i}f(t,x,v)&=\dfrac{1}{2}(\partial_{x^i}-\partial_{v^i})f(t,x,v)+\dfrac{1}{2}(\partial_{x^i}+\partial_{v^i})f(t,x,v)\\
&=\dfrac{1}{2}e^t(\partial_{x^i}-\partial_{v^i})f_0(x_0,v_0)+\dfrac{1}{2}e^{-t}(\partial_{x^i}+\partial_{v^i})f_0(x_0,v_0),   
\end{align*}
in terms of a point $(x_0,v_0)$ in the support of the initial distribution function $f_0$.
\end{remark}

Let us also consider the macroscopic vector fields associated to the microscopic vector fields previously defined by 
\begin{equation*}
    U_i^{x}=e^t\partial_{x^i},\quad S_i^{x}=e^{-t}\partial_{x^i}, \quad 
    L^x=\sum_{i=1}^n x^i\partial_{x^i},\quad R_{ij}^x=x^i\partial_{x^j}-x^j\partial_{x^i},
\end{equation*}
and define 
\begin{equation*}
\Lambda=\Big\{ U^x_i, S^x_i, L^x, R^x_{ij}  \Big\},\quad  \Lambda_0=\Big\{ U^x_i, L^x, R^x_{ij}  \Big\},
\end{equation*}
for $i,j\in\{1,\ldots,n\}$. The set of macroscopic vector fields $\Lambda$ and the set of microscopic vector fields $\lambda$ are precisely related to each other by the following result for the study of the spatial density of an arbitrary distribution function. 

\begin{lemma}\label{lemma_linear_vlasov_derivatives_spatial_density}
Let $f$ be a sufficiently regular distribution function. Then, the derivatives of the induced spatial density satisfy
\begin{align*}
    U_i^x\rho(f)&=\rho(U_if),\qquad L^x\rho(f)=\rho(Lf)+n\rho(f), \\
    S_i^x\rho(f)&=\rho(S_{i}f),\qquad R_{ij}^x\rho(f)=\rho(R_{ij}f)
\end{align*}
for every $i,j\in\{1,2,\dots,n\}$
\end{lemma}

\subsection{Macroscopic and microscopic differential operators}

Let $(Z^i)_i$ be an arbitrary ordering of the microscopic vector fields contained in $\lambda$. In the following, we use a multi-index notation for the microscopic differential operators of order $|\alpha|$ given by the composition $$Z^{\alpha}:=Z^{\alpha_1} Z^{\alpha_2}\dots Z^{\alpha_{n}},$$ for every multi-index $\alpha\in \N^{n}$. We denote by $\lambda^{|\alpha|}$ the family of microscopic differential operators obtained as a composition of $|\alpha|$ vector fields in $\lambda$. Furthermore, we can uniquely associate a macroscopic differential operator to any microscopic differential operator $Z^{\alpha}\in \lambda^{|\alpha|}$ by replacing every microscopic vector field $Z$ by the corresponding macroscopic vector field $Z^x$. By a small abuse of notation, we denote also by $Z^{\alpha}$ the associated macroscopic differential operator to an arbitrary microscopic differential operator $Z^{\alpha}$. We denote by $\Lambda^{|\alpha|}$ the family of macroscopic differential operators of order $|\alpha|$ obtained as a composition of $|\alpha|$ vector fields in $\Lambda$. Finally, we denote by $\partial_x^{\alpha}$ a standard macroscopic differential operator $$\partial_x^{\alpha}:=\partial^{\alpha_1}_{x^{1}}\partial^{\alpha_2}_{x^2}\dots \partial^{\alpha_n}_{x^n},$$ for every multi-index $\alpha\in \N^{n}$.

In the following, we prove that the arbitrary ordering of the vector fields chosen to build differential operators can be taken without loss of generality modulo some uniform constants.

\begin{lemma}\label{lemma_commuting_in_lambda}
Let $\Omega\in \{\lambda, \lambda_0, \Lambda,\Lambda_0\}$. Let $\alpha$ and $\beta$ be two multi-indices. Then, the commutator between $Z^{\alpha}\in \Omega^{|\alpha|}$ and $Z^{\beta}\in\Omega^{|\beta|}$ is given by $$[Z^{\alpha}, Z^{\beta}]=\sum_{|\gamma|\le |\alpha|+|\beta|-1}\sum_{Z^\gamma\in \Omega^{|\gamma|}}C^{\alpha \beta}_{\gamma}Z^{\gamma},$$
for some constant coefficients $C^{\alpha \beta}_{\gamma}$.
\end{lemma}
\begin{proof}
Observe that  $$[U_i,R_{ij}]=U_j,\quad [S_i,R_{ij}]=S_j, \quad [L, R_{ij}]=0,\quad [R_{ij},R_{jk}]=R_{ik},$$
$$[U_i,L]=U_i, \quad [S_i,L]=S_i, \quad [U_i, S_j]=0, \quad [U_i,U_j]=0, \quad [S_i,S_j]=0,$$
for $i,j\in\{1,\ldots,n\}$, and note that the same commutation relations hold if we replace $Z\in \lambda$ by the associated macroscopic vector fields $Z^x\in \Lambda$. This argument proves the result for $|\alpha|=|\beta|=1$. The general statement follows by induction.  
\end{proof}

Moreover, we can use the microscopic differential operators previously discussed to build conservation laws for higher order derivatives of a sufficiently regular solution of the Vlasov equation (\ref{vlasov_linear_flow}) as in Corollary \ref{corollary_conservation_law_vector_fields}.

\begin{corollary}\label{corollary_conservation_law_differential_operators}
Let $f_0$ be a sufficiently regular initial data for the Vlasov equation (\ref{vlasov_linear_flow}). Then, the corresponding solution $f$ to the Vlasov equation (\ref{vlasov_linear_flow}) satisfies $$\|Z^{\alpha}f(t)\|_{L^1_{x,v}}=\|Z^{\alpha} f(0)\|_{L^1_{x,v}},$$ for every $t\geq 0$, and every multi-index $\alpha$.
\end{corollary}

In the following, we state a key vector field identity to obtain quantitative decay estimates in space and time for the spatial density induced by an arbitrary distribution function in terms of a higher order energy norm according to the weighted Sobolev inequalities proven in the following section. For this purpose, we firstly recall the relation $$|x|^2 \partial_{x^j}=\sum_{i=1}^nx^iR_{ij}^x + x^j L^x,$$ noticed in \cite[Lemma 2.5]{Sm16} between the macroscopic rotations and the macroscopic scaling. As a result, we have $$|x| \partial_{x^j}=\sum_{i=1}^n\frac{x^i}{|x|}R_{ij}^x + \frac{x^j}{|x|} L^x,$$ which allows to prove the following useful lemma.

\begin{lemma}\label{lemma_linear_vlasov_weight_vector_fields}
For any multi-index $\alpha$, we have 
\begin{equation}\label{identity_macroscopic_giving_decay}
    (e^t+|x|)^{\alpha}\partial_x^{\alpha}=\sum_{|\beta|\leq |\alpha|}\sum_{Z^{\beta}\in \Lambda_0^{|\beta|}} C_{\beta}Z^{\beta},
\end{equation}
for some uniformly bounded functions $C_{\beta}$.
\end{lemma}

We conclude this subsection by relating the macroscopic and microscopic differential operators in the same manner as in Lemma \ref{lemma_linear_vlasov_derivatives_spatial_density}. 

\begin{lemma}\label{lemma_connection_microscopic_macroscopic_vector_fields}
Let $f$ be a sufficiently regular distribution function and let $\alpha$ be a multi-index. Then, there exist constant coefficients $C^{\alpha}_{\beta}$ such that
\begin{equation}
    Z^{\alpha}\rho(f)=\rho(Z^{\alpha}f)+\sum_{|\beta|\leq |\alpha|-1}C^{\alpha}_{ \beta}\rho(Z^{\beta}f),
\end{equation}
where the vector fields in the left hand side are macroscopic, whereas the ones in the right hand side are microscopic.
\end{lemma}

\subsection{The commuted equations}
Let us denote the non-linear transport operator applied to the distribution function in the Vlasov--Poisson system with the external potential $\frac{-|x|^2}{2}$ by $$\T_\phi :=\partial_t +v\cdot\nabla_x +x\cdot\nabla_v -\mu\nabla_x\phi\cdot \nabla_v ,$$ where the field $\nabla_x\phi$ is defined through the Poisson equation $\Delta \phi=\rho(f)$.

\begin{lemma}\label{lemma_commuted_nonliner_Vlasov}
There exist constant coefficients $C^{\alpha}_{\beta \gamma}$ such that \begin{equation}\label{eq_comm}
    [\T_{\phi},Z^{\alpha}]=\sum_{|\gamma|+|\beta|\leq |\alpha|,}\sum_{|\beta|\leq |\alpha|-1} C^{\alpha}_{\beta \gamma}\nabla_x Z^{\gamma}\phi\cdot  \nabla_v Z^{\beta},
\end{equation}
where the vector fields $Z^{\alpha}\in\lambda^{|\alpha|}$, $Z^{\gamma}\in \Lambda^{|\gamma|}$, and $Z^{\beta}\in \lambda^{|\beta|}$.
\end{lemma}

\begin{proof}
For each vector field $Z^i\in \lambda$, we can easily compute 
$$[\T_{\phi},Z^i]=\mu \sum_{k=1}^n \partial_{x^k}(Z^i\phi+c_i\phi)\partial_{v^k},$$ where $c_i=-2$ if $Z^i=L$, otherwise, $c_i=0$. This verifies equation (\ref{eq_comm}) for $|\alpha|=1$. We argue inductively on $|\alpha|$ to prove the general case. Observe that
$$[\T_\phi, Z^iZ^\alpha]=[\T_\phi, Z^i]Z^\alpha+Z^i[\T_\phi, Z^\alpha].$$
Since $[\T_\phi, Z^i]Z^\alpha$ has the required form, it remains to analyse the second term. Note that 
\begin{align*}
Z^i[\T_\phi, Z^\alpha]&= \sum_{\beta,\gamma}C^{\alpha}_{\beta \gamma} \sum_{k=1}^n  Z^i(\partial_{x^k} Z^{\gamma}\phi ) \partial_{v^k} Z^{\beta}+\partial_{x^k} Z^{\gamma}\phi Z^i( \partial_{v^k} Z^{\beta})\\
&=  \sum_{\beta,\gamma}C^{\alpha}_{\beta \gamma}\sum_{k=1}^n  \partial_{x^k} (Z^i Z^{\gamma}\phi) \partial_{v^k} Z^{\beta}+\partial_{x^k} Z^{\gamma}\phi \partial_{v^k} Z^i Z^{\beta}\\
&\qquad+  \sum_{\beta,\gamma}C^{\alpha}_{\beta \gamma}\sum_{k=1}^n   [Z^i, \partial_{x^k}] Z^{\gamma}\phi \partial_{v^k} Z^{\beta}+\partial_{x^k} Z^{\gamma}\phi [Z^i,\partial_{v^k}] Z^{\beta},
\end{align*}
where we have applied $Z^i$ to equation \eqref{eq_comm}, which is our inductive assumption. In the last equality the first term has the correct form and the second one behaves nicely for all choices of $Z^i$: if $Z^i$ is stable or unstable, then the commutators vanish; if $Z^i$ is a rotation the summation cancel out; and if $Z^i=L$ we note that $[L^i,\partial_x^i]=-\partial_x^i$, $[L^i,-\partial_v^i]=-\partial_v^i$. Therefore, the sum has always the required form.
\end{proof}

\begin{lemma}\label{lemma_commuted_poisson_equation}
Let $f$ be a sufficiently regular distribution function, and let $\phi$ be the
solution to the Poisson equation $\Delta \phi = \rho(f)$. Then, for any multi-index $\alpha$ the function $Z^{\alpha}\phi$ satisfies the equation $$\Delta Z^{\alpha}\phi=\sum_{|\beta|\leq |\alpha|}C^{\alpha}_{\beta}Z^{\beta}\rho(f),$$ for some constant coefficients $C^{\alpha}_{\beta}$.
\end{lemma}
\begin{proof}
Note that $[\Delta,Z]=0$ for any $Z\in\Lambda\setminus\{L^x\}$, and that  $[\Delta,L^x]=2\Delta$. For $|\alpha|=1$ the result holds trivially. For higher order derivatives we proceed by induction and use that 
\begin{align*}
    \Delta Z^i Z^\alpha\phi&=Z^i \Delta Z^\alpha\phi+[\Delta,Z^i]Z^\alpha\phi,
\end{align*}
noticing that $[\Delta,Z^i]$ is either equal to zero, or to a multiple of $\Delta.$
\end{proof}

\section{Decay of velocity averages for the linearized system}\label{section_exponential_decay_velocity_averages}

In this section, we begin by proving weighted Sobolev inequalities for arbitrary finite energy distribution functions by exploiting the weights contained in the set of macroscopic vector fields $\Lambda_0$ and the set of microscopic vector fields $\lambda_0$. As a result, we prove sharp quantitative decay estimates in space and time for the spatial density induced by solutions to the linear Vlasov equation (\ref{vlasov_linear_flow}). We also obtain improved decay estimates for derivatives of the spatial density. 

\subsection{Weighted Sobolev inequalities}

First, we prove a weighted Sobolev inequality for the spatial density induced by arbitrary finite energy distribution functions. 

\begin{proposition}\label{proposition_weighted_sobolev_linear_vlasov}
For every sufficiently regular distribution function $f$, the induced spatial density satisfies that 
\begin{equation}\label{prop_1}
    |\rho(f)(t,x)|\lesssim \dfrac{1}{(e^t+|x|)^{n}}\sum_{|\alpha|\leq n}\sum_{Z^{\alpha}\in \lambda_0^{|\alpha|}} \|Z^{\alpha}f\|_{L^1_{x,v}},
\end{equation}
for every $t\geq0$ and every $x\in \R^n$.
\end{proposition}

\begin{proof}
Given a point $(t,x)\in \R\times \R^{n}$, we set the function $\widetilde{\rho}: \R^{n}\to \R$ given by $\widetilde{\rho}(y)=\rho(f)(t,x+(e^t+|x|)y)$. Applying the standard Sobolev inequality, we have that \begin{align}\label{sobolev}|\rho(f)(t,x)|=|\widetilde{\rho}(0)|\leq \sum_{|\alpha|\leq n}\|\partial_y^{\alpha}\widetilde{\rho}\|_{L^1(B_{n}(0,1/2))},
\end{align}
where $B_{n}(0,1/2)$ denotes the open ball in $\R^{n}_y$ of radius $1/2$. By the chain rule, we have that $$\partial_{y^j}\widetilde{\rho}(y)=(e^t+|x|)\partial_{x^j}\rho(f)(t,x+(e^t+|x|)y).$$ Hence, the derivatives $\partial_y^{\alpha}\widetilde{\rho}$ can be bounded for every $y\in B_{n}(0,1/2)$ and $|\alpha|\le n$ by
\begin{align*}
    |\partial_y^{\alpha}\widetilde{\rho}(y)|& =(e^t+|x|)^{|\alpha|}|\partial_x^{\alpha}\rho(f)(t,x+(e^t+|x|)y)|\\
    & \lesssim (e^t+|x+(e^t+|x|)y|)^{|\alpha|}|\partial_x^{\alpha}\rho(f)(t,x+(e^t+|x|)y)|\\
    &\lesssim \sum_{|\beta|\leq |\alpha|}\sum_{Z^{\beta}\in \Lambda_0^{|\beta|}}|Z^{\beta}\rho(f)(t,x+(e^t+|x|)y)|,
\end{align*}
where in the second inequality we have compared $\min_{y\in B_{n}(0,1/2)} e^t+|x+(e^t+|x|)y|$ with $e^t+|x|$, and in the last inequality we have used Lemma \ref{lemma_linear_vlasov_weight_vector_fields}. Integrating in the $y$ coordinate, applying the change of variables $z=(e^t+|x|)y$, and using the Sobolev inequality (\ref{sobolev}) we obtain 
\begin{equation}\label{inequality_rho}
    |\rho(f)(t,x)|\lesssim \dfrac{1}{(e^t+|x|)^{n}}\sum_{|\beta|\leq n}\sum_{Z^{\beta}\in \Lambda_0^{|\beta|}} \|Z^{\beta}\rho(f)\|_{L^1_{x}}.
\end{equation}
Finally, we use Lemma \ref{lemma_connection_microscopic_macroscopic_vector_fields} to conclude the proof of the proposition.
\end{proof}

We proceed to prove another weighted Sobolev inequality for the spatial density induced by \emph{absolute values} of arbitrary finite energy distribution functions. The proof follows by a slightly different argument as the one obtained for Proposition \ref{proposition_weighted_sobolev_absolute_values}.

\begin{proposition}\label{proposition_weighted_sobolev_absolute_values}
For every sufficiently regular distribution function $f$, the induced spatial density by its absolute value satisfies that
\begin{equation}
        \rho(|f|)(t,x)\lesssim \dfrac{1}{(e^t+|x|)^{n}}\sum_{|\alpha|\leq n}\sum_{Z^{\alpha}\in \lambda_0^{|\alpha|}} \|Z^{\alpha}f\|_{L^1_{x,v}},
\end{equation}
for every $t\geq0$ and every $x\in \R^n$.
\end{proposition}

\begin{proof}
Similarly as in the proof of Proposition \ref{proposition_weighted_sobolev_linear_vlasov}, we define a real-valued function $\widetilde{\psi}:B_{n}(0,1/2)\to \R$ given by $$\widetilde{\psi}(y):=\int_{\R^{n}}|f|\Big(t,x+(e^t+|x|)y,v\Big)dv.$$ Using a 1D Sobolev inequality with $\delta=\frac{1}{4n}$, we have
\begin{equation}
    \widetilde{\psi}(0)\leq C\int_{|y_1|\leq \delta^{1/2}} |\partial_{y_1}\widetilde{\psi}(y_1,0\dots,0)|+|\widetilde{\psi}(y_1,0\dots,0)|dy_1,
\end{equation}
where we have used that for a function $\psi\in W^{1,1}$, the absolute value of $\psi$ belongs to $W^{1,1}$, and satisfies $|\partial |\psi||\leq |\partial \psi|$. Moreover, the derivative in the previous integral can be written as
\begin{align*}
    \partial_{y_1}\widetilde{\psi}(y_1,0,\dots,0)&=\int_{\R^{n}} e^t\partial_{x^1}|f|\Big(t,x+(e^t+|x|)(y_1,0,\dots,0),v\Big)dv\\
    &\qquad +\int_{\R^{n}} |x|\partial_{x^1}|f|\Big(t,x+(e^t+|x|)(y_1,0,\dots,0),v\Big)dv.
\end{align*}
The first integral term of the derivative above can be estimated using integration by parts in the velocity variables to obtain
\begin{align*}
    |\partial_{y_1}\widetilde{\psi}(y_1,0,\dots,0)|&\leq \Big|\int_{\R^{n}} e^t\partial_{x^1}|f|\Big(t,x+(e^t+|x|)(y_1,0,\dots,0),v\Big)dv\Big|\\
    &\leq \int_{\R^{n}}\Big|e^t(\partial_{x^1}+\partial_{p^1})f\Big(t,x+(e^t+|x|)(y_1,0,\dots,0),v\Big)\Big|dv,
\end{align*}
and similarly for the second integral term of the derivative above. As a result, we have that 
\begin{equation}
    |\partial_{y_1}\widetilde{\psi}(y_1,0,\dots,0)|\leq \sum_{Z\in \lambda}\int_{\R^{n}}\Big|Zf\Big(t,x+(e^t+|x|)(y_1,0,\dots,0),v\Big)\Big|dv,
\end{equation}
which can be used to estimate $\widetilde{\psi}$ by 
\begin{align*}
\widetilde{\psi}(0)&\leq \sum_{Z\in \lambda}\int_{|y_1|\leq \delta^{1/2}}\int_{\R^{n}}\Big|Zf\Big(t,x+(e^t+|x|)(y_1,0,\dots,0),v\Big)\Big|dvdy_1 \\
&\quad+ \int_{|y_1|\leq \delta^{1/2}}\int_{\R^{n}}\Big|f\Big(t,x+(e^t+|x|)(y_1,0,\dots,0),v\Big)\Big|dvdy_1.
\end{align*}
Iterating this argument for all the variables in space, we obtain
\begin{equation}
    \widetilde{\psi}(0)\leq \sum_{Z^{\alpha}\in \lambda^{|\alpha|}}\int_{y\in B_{n}(0,1/2)}\int_{\R^{n}}\Big|Z^{\alpha}f\Big(t,x+(e^t+|x|)y,v\Big)\Big|dvdy,
\end{equation}
from which the proof of the proposition follows using the change of variables $z=(t+|x|)y$.
\end{proof}

Finally, we obtain improved decay estimates for derivatives of the spatial density by applying the weighted Sobolev inequality in Proposition \ref{proposition_weighted_sobolev_linear_vlasov} combined with Lemma \ref{lemma_linear_vlasov_derivatives_spatial_density} and Lemma \ref{lemma_linear_vlasov_weight_vector_fields}. 

\begin{proposition}[Improved decay estimates for derivatives of the spatial density]\label{proposition_improved_decay_spatial_density}
For every sufficiently regular distribution function $f$, the induced spatial density satisfies that  
\begin{equation}
    |\partial^{\alpha}_x\rho(f)(t,x)|\lesssim \dfrac{1}{(e^t+|x|)^{n+|\alpha|}}\sum_{|\beta|\leq n+|\alpha|}\sum_{Z^{\beta}\in \lambda_0^{|\beta|}} \|Z^{\beta}f\|_{L^1_{x,v}},
\end{equation}
for every $t\geq 0$, every $x\in \R^{n}$, and every multi-index $\alpha$.
\end{proposition}

\begin{proof}
Applying the estimate (\ref{inequality_rho}) obtained in the proof of Proposition \ref{proposition_weighted_sobolev_linear_vlasov}, we have that
\begin{equation}
    |\partial_x^{\alpha}\rho(f)(t,x)|\lesssim \dfrac{1}{(e^t+|x|)^{n}}\sum_{|\beta|\leq n}\sum_{Z^{\beta}\in \Lambda_0^{|\beta|}} \|Z^{\beta}\partial_x^{\alpha}\rho(f)\|_{L^1_x}. 
\end{equation}
The improved decay for derivatives of the spatial density follows by commuting the differential operators $Z^{\beta}$, $\partial_x^{\alpha}$, and using Lemma \ref{lemma_linear_vlasov_weight_vector_fields}.
\end{proof}

\subsection{Applications to solutions to the Vlasov equation with the potential \texorpdfstring{$\frac{-|x|^2}{2}$}{x2}}

The weighted Sobolev inequality in Proposition \ref{proposition_weighted_sobolev_linear_vlasov} shows that the spatial density induced by solutions to the linear Vlasov equation (\ref{vlasov_linear_flow}) decay quantitatively in space and time.

\begin{corollary}\label{corollary_decay_linear_vlasov}
Let $f_0$ be a sufficiently regular initial data for the Vlasov equation (\ref{vlasov_linear_flow}). Then, the induced spatial density for the corresponding solution $f$ to the Vlasov equation with the trapping potential $\frac{-|x|^2}{2}$ satisfies \begin{equation}
        |\rho(f)(t,x)|\lesssim \dfrac{1}{(e^t+|x|)^{n}}\sum_{|\alpha|\leq n}\sum_{Z^{\alpha}\in \lambda_0^{|\alpha|}} \|Z^{\alpha}f_0\|_{L^1_{x,v}},
\end{equation}
for every $t\geq 0$, and every $x\in \R^{n}$.
\end{corollary}

The weighted Sobolev inequality in Proposition \ref{proposition_weighted_sobolev_absolute_values} shows that the spatial density induced by the absolute value of solutions to the linear Vlasov equation (\ref{vlasov_linear_flow}) decay quantitatively in space and time.

\begin{corollary}\label{corollary_decay_linear_vlasov_absolute_values}
Let $f_0$ be a sufficiently regular initial data for the Vlasov equation (\ref{vlasov_linear_flow}). Then, the induced spatial density for the corresponding solution $f$ to the Vlasov equation with the trapping potential $\frac{-|x|^2}{2}$ satisfies \begin{equation}
        \rho(|f|)(t,x)\lesssim \dfrac{1}{(e^t+|x|)^{n}}\sum_{|\alpha|\leq n}\sum_{Z^{\alpha}\in \lambda_0^{|\alpha|}} \|Z^{\alpha}f_0\|_{L^1_{x,v}},
\end{equation}
for every $t\geq 0$, and every $x\in \R^{n}$.
\end{corollary}

Finally, the improved decay estimates for derivatives of the spatial density in Proposition \ref{proposition_improved_decay_spatial_density} shows that the derivatives of the spatial density induced by solutions to the linear Vlasov equation (\ref{vlasov_linear_flow}) decay quantitatively in space and time.

\begin{corollary} \label{corollary_improved_decay_spatial_density}
Let $f_0$ be a sufficiently regular initial data for the Vlasov equation (\ref{vlasov_linear_flow}). Then, the induced spatial density for the corresponding solution $f$ to the Vlasov equation with the trapping potential $\frac{-|x|^2}{2}$ satisfies 
\begin{equation}
    |\partial^{\alpha}_x\rho(f)(t,x)|\lesssim \dfrac{1}{(e^t+|x|)^{n+|\alpha|}}\sum_{|\beta|\leq n+|\alpha|}\sum_{Z^{\beta}\in \lambda_0^{|\beta|}} \|Z^{\beta}f_0\|_{L^1_{x,v}},
\end{equation}
for every $t\geq 0$, every $x\in \R^{n}$, and every multi-index $\alpha$.
\end{corollary}

\section{Small data solutions for the Vlasov--Poisson system with the potential \texorpdfstring{$\frac{-|x|^2}{2}$}{x2}}\label{section_proof_theorem_dimension_higher_two}

In this section, we study the evolution in time of sufficiently regular small data solutions $f$ for the Vlasov--Poisson system with the trapping potential $\frac{-|x|^2}{2}$ in the energy space defined by the norm
$$\mathcal{E}_N[f]:=\sum_{|\alpha|\leq N}\sum_{ Z^{\alpha}\in {\lambda}^{|\alpha|}}\|Z^{\alpha}f\|_{L^1_{x,v}},$$ where $N\in\N$. We emphasize that this energy norm is stronger than the energy norms used to obtain weighted Sobolev inequalities in the previous section. Nonetheless, we can still use this norm to prove quantitative decay estimates for the spatial density induced by solutions of the non-linear system, a crucial ingredient of the global existence result. More precisely, we have included the vector fields contained in the stable invariant distribution of phase space in the energy norm used in this section. We incorporate these vector fields, as together with the unstable vector fields, they generate the standard basis $\{\partial_{x^i},\partial_{v^i}\}$ of the tangent space of $\R^n_x\times \R^n_v.$ We make use of this fact in the proof of Theorem \ref{theorem_stability_vacuum_external_potential}. 

\subsection{The bootstrap assumption}
The proof of Theorem \ref{theorem_stability_vacuum_external_potential} follows by a standard continuity argument. We aim to prove that for $\e>0$ sufficiently small, if the initial data satisfies $\mathcal{E}_N[f_0]\le \e$, then, the global energy estimate $\mathcal{E}_N[f(t)]\le 2\e$ holds for every $t\ge 0$. For this purpose, we define
\begin{equation}\label{bootstrap_assumption_energy_estimate}
T:=\sup\Big\{t\ge 0:\mathcal{E}_N[f(s)]\le 2\e \text{ for every } s\in [0,t]\Big\}.    
\end{equation}
In the following, we show that the energy of the distribution function satisfies $\mathcal{E}_N[f(t)]\leq \frac{3\e}{2}$ for every $t\in [0,T]$. Therefore, the supremum (\ref{bootstrap_assumption_energy_estimate}) is infinite, and we obtain global existence of small data solutions for the Vlasov--Poisson system with the trapping potential $\frac{-|x|^2}{2}$.

\subsection{Proof of Theorem \ref{theorem_stability_vacuum_external_potential}}

By the standard energy estimate for the commuted distribution function $Z^{\alpha}f$ in $L^1_{x,v}$, we obtain 
\begin{equation}
    \|Z^{\alpha}f(t)\|_{L^1_{x,v}}\leq \|Z^{\alpha}f(0)\|_{L^1_{x,v}} + \int_0^t \|\T_{\phi}(Z^{\alpha}f)(s)\|_{L^1_{x,v}}ds.
\end{equation}
Furthermore, we write the non-linear Vlasov equation for the commuted distribution function $Z^{\alpha}f$ as
\begin{equation}
    \T_{\phi}(Z^{\alpha}f)=\sum_{|\beta|\leq |\alpha|-1,}\sum_{|\gamma|+|\beta|\leq |\alpha|}C^{\alpha}_{\beta \gamma} \nabla_x Z^{\gamma}\phi \cdot \nabla_v Z^{\beta}f,
\end{equation}
by using the commutator in Lemma \ref{lemma_commuted_nonliner_Vlasov}. We write the gradient in the velocity variables on the previous identity using the stable and unstable vector fields contained in $\lambda$ by
\begin{equation}
    \partial_{v^i} Z^{\beta}f=\dfrac{1}{2e^t}\Big(e^t(\partial_{x^i}+\partial_{v^i})Z^{\beta}f\Big)-\dfrac{e^t}{2}\Big(e^{-t}(\partial_{x^i}-\partial_{v^i})Z^{\beta}f\Big),
\end{equation}
to obtain the bound 
\begin{equation}\label{estimate_Vlasov_for_commuted_distribution_ growing_factor}
    \|\T_{\phi}(Z^{\alpha}f)\|_{L^1_{x,v}}\lesssim e^t \Big(\sum_{1\leq |\beta|\leq |\alpha|,}\sum_{|\gamma|+|\beta|\leq |\alpha|+1}\|\nabla_x Z^{\gamma}(\phi)   Z^{\beta}f\|_{L^1_{x,v}}\Big)
\end{equation}
for the non-linear contribution in the energy estimate for the commuted distribution function $Z^{\alpha}f$. In order to bound the non-linear terms $\|\nabla_xZ^\gamma(\phi)Z^\beta f\|_{L^1_{x,v}}$, we follow the strategy used by Duan \cite{Du22} to prove the stability of the vacuum solution for the Vlasov--Poisson system. More precisely, we make use of the explicit form of the Green function for the Poisson equation in $\R^n$ to estimate the gradient $\nabla_x Z^\gamma\phi$ combined with the bootstrap assumption to bound the derivatives of the distribution function $Z^{\beta}f$ in $L^1_{x,v}$. For this purpose, we need the following elementary estimate proved in \cite[Lemma 3.2]{Du22}. 

\begin{lemma}\label{lemma_uniform_integral_bound_kernel_convolution_duan}
For every $n\geq 2$, there exists a uniform constant $C_n>0$ depending only on $n$, such that for every $x\in \R^n$ we have
\begin{equation*}
    \int_{\R^n}\dfrac{dy}{|y|^{n-1}(1+|x+y|)^n} \leq C_n.
\end{equation*}
\end{lemma}

As a consequence of Lemma \ref{lemma_uniform_integral_bound_kernel_convolution_duan}, we obtain decay in time for the integral term
\begin{equation}\label{remark_uniform_integral_bound_kernel_convolution_duan}
    \int_{\R^n}\dfrac{1}{|y|^{n-1}(e^t+|x-y|)^n} dy=\dfrac{1}{e^{(n-1)t}}\int_{\R^n} \dfrac{1}{|y'|^{n-1}(1+|y'-\frac{x}{e^t}|)^n}dy'\lesssim \dfrac{1}{e^{(n-1)t}},
\end{equation}
by using the change of variables $y=e^t y'$. We use the estimate (\ref{remark_uniform_integral_bound_kernel_convolution_duan}) to prove decay for the gradient $\nabla_x Z^\gamma\phi$. We improve the bootstrap assumption (\ref{bootstrap_assumption_energy_estimate}) using the following technical lemma to bound the non-linear terms $\|\nabla_xZ^\gamma(\phi)Z^\beta f\|_{L^1_{x,v}}$.

\begin{lemma}\label{lemma_decay_bilinear_terms_commuted_vlasov}
Under the bootstrap assumption (\ref{bootstrap_assumption_energy_estimate}), the corresponding solution $f$ to the Vlasov--Poisson system with the potential $\frac{-|x|^2}{2}$ satisfies 
\begin{equation}\label{estimate_bilinear_terms_commuted_vlasov_statement}
    \|\nabla_x Z^{\gamma}(\phi)   Z^{\beta}f\|_{L^1_{x,v}}\lesssim \dfrac{\epsilon^2}{e^{(n-1)t}}
\end{equation}
for every $t\in [0,T]$, and for any multi-indices $\beta$, $\gamma$ such that $|\beta|\leq N$, $|\gamma|\leq N$, and $|\beta|+|\gamma|\leq N+1$.
\end{lemma}

\begin{proof}
Combining the commuted Poisson equation in Lemma \ref{lemma_commuted_poisson_equation} with the relation between the macroscopic and microscopic vector fields established in Lemma \ref{lemma_connection_microscopic_macroscopic_vector_fields}, we obtain 
$$\Delta Z^{\gamma}\phi=\sum_{|\gamma'|\leq |\gamma|}C^{\gamma}_{\gamma'}\rho(Z^{\gamma'}f),$$ for some fixed coefficients $C^{\gamma}_{\gamma'}$. We use the Green function for the Poisson equation in $\R^n$ to write the solution of the commuted Poisson equation as $$Z^{\gamma}\phi(t,x)=\sum_{|\gamma'|\leq |\gamma|}\int_{\R^n} \dfrac{C_nC_{\gamma'}^{\gamma}}{|y|^{n-2}}\rho(Z^{\gamma'}f)(t,x-y)dy,$$ whose gradient can be estimated directly by 
\begin{equation}\label{nabla_phi}
|\nabla_x Z^{\gamma}\phi(t,x)|\lesssim \sum_{|\gamma'|\leq |\gamma|}\int_{\R^n} \dfrac{1}{|y|^{n-1}}\rho(|Z^{\gamma'}f|)(t,x-y)dy.\end{equation}
By the weighted Sobolev inequality for the absolute value of distribution functions in Proposition \ref{proposition_weighted_sobolev_absolute_values}, we estimate
\begin{align*}
    \rho(|Z^{\gamma'}f|)(t,x-y)&\lesssim \dfrac{1}{(e^t+|x-y|)^n}\sum_{|\beta''|\leq |\gamma'|+n}\|Z^{\beta''}f\|_{L^1_{x,v}}\\
    &\lesssim \dfrac{1}{(e^t+|x-y|)^n}\sum_{|\beta''|\leq N}\|Z^{\beta''}f\|_{L^1_{x,v}}\\
    &\lesssim \dfrac{\epsilon}{(e^t+|x-y|)^n},
\end{align*}
for every $|\gamma'|\leq N-n$. Hence, the solution of the commuted Poisson equation satisfies that for every $|\gamma|\leq N-n$, we have $$|\nabla_x Z^{\gamma}\phi(t,x)|\lesssim \e \sum_{|\gamma'|\leq |\gamma|}\int_{\R^n} \dfrac{1}{|y|^{n-1}(e^t+|x-y|)^n}dy\lesssim \frac{\e}{e^{(n-1)t}},$$ 
where we have used the estimate (\ref{remark_uniform_integral_bound_kernel_convolution_duan}) in the last inequality. As a result, the left hand side of (\ref{estimate_bilinear_terms_commuted_vlasov_statement}) can be bounded by
\begin{align*}
    \|\nabla_x Z^{\gamma}(\phi)   Z^{\beta}f\|_{L^1_{x,v}}&\lesssim \frac{\e}{e^{(n-1)t}}\|Z^\beta f\|_{L^1_{x,v}}\\
    &\lesssim \frac{\e}{e^{(n-1)t}}\sum_{|\beta|\le N}\|Z^\beta f\|_{L^1_{x,v}}\\
    &\lesssim \frac{\epsilon^2}{e^{(n-1)t}},
\end{align*}
for every $|\gamma|\leq N-n$. Otherwise, if $|\gamma|> N-n$ then $|\beta|\leq N-n$, since $|\beta|+|\gamma|\leq N+1$ and $N\geq 2n$. It follows from the bound $|\beta|\le N-n$ and Proposition \ref{proposition_weighted_sobolev_absolute_values} that
\begin{align}\label{last_1}
\rho(|Z^\beta f|)(t,x)\lesssim \frac{\e}{(e^t+|x|)^n}.
\end{align}
Therefore 
\begin{align*}
    \label{1}
    \|\nabla_x Z^\gamma (\phi) Z^\beta f\|_{L^1_{x,v}} &= \int|\nabla_x Z^\gamma \phi (t,x)|\rho(|Z^\beta f|)(t,x)dx \\
    & \lesssim \e \sum_{|\gamma'|\le |\gamma|}\iint \frac{1}{|y|^{n-1}} \rho(|Z^{\gamma'}f|)(t,x-y)\frac{1}{(e^t+|x|)^n} dx dy,\\
    &\lesssim \e \sum_{|\gamma'|\le |\gamma|}\iint \frac{1}{|y|^{n-1}} \rho(|Z^{\gamma'}f|)(t,z)\frac{1}{(e^t+|z+y|)^n} dz dy,\\
    &\lesssim \e \sum_{|\gamma'|\le |\gamma|} \int \rho(|Z^{\gamma'}f|)(t,z) \bigg(\int \frac{1}{|y|^{n-1}(e^t+|z+y|)^n} dy\bigg)dz\\
    &\lesssim \frac{\e}{e^{(n-1)t}}\sum_{|\gamma'|\le N}\|Z^{\gamma'} f\|_{L^1_{x,v}}\\
    &\lesssim\frac{\e^2}{e^{(n-1)t}},
\end{align*}
where we have used the change of variables $z=x-y$ and the previous estimates (\ref{remark_uniform_integral_bound_kernel_convolution_duan}), (\ref{nabla_phi}), and (\ref{last_1}).
\end{proof}

The quantitative decay estimate for the non-linear terms $\nabla_x Z^{\gamma}(\phi) Z^{\beta}f$ given by Lemma \ref{lemma_decay_bilinear_terms_commuted_vlasov} shows that the $L^1_{x,v}$ norm of the non-linear contribution in the energy estimate for the commuted distribution function $Z^{\alpha}f$ satisfies that for every $t\in [0,T]$ we have
\begin{align*}
\| \T_{\phi}(Z^{\alpha}f)\|_{L^1_{x,v}}&\lesssim e^t \Big(\sum_{1\leq |\beta|\leq |\alpha|,}\sum_{|\gamma|+|\beta|\leq |\alpha|+1}\|\nabla_x Z^{\gamma}(\phi)   Z^{\beta}f\|_{L^1_{x,v}}\Big)\\
&\lesssim \dfrac{\epsilon^2}{e^{(n-2)t}}, 
\end{align*}
by using the bound (\ref{estimate_Vlasov_for_commuted_distribution_ growing_factor}) previously obtained. Therefore, the energy $\mathcal{E}_N[f]$ of the solution to the Vlasov--Poisson system with the potential $\frac{-|x|^2}{2}$ is bounded for every $t\in [0,T]$ by
\begin{equation}\label{estimate_energy_improve_bootstrap_assumption}
    \mathcal{E}_N[f(t)]\leq \mathcal{E}_N[f(0)]+C\epsilon^2\int_0^t \dfrac{ds}{e^{(n-2)s}},
\end{equation}
where $C>0$ is a uniform constant depending only on $n$ and $N$. We emphasize that the time integral in the right hand side of (\ref{estimate_energy_improve_bootstrap_assumption}) is uniformly bounded for any $t\geq 0$, due to the exponential decay in time of $\exp(-(n-2)t)$ in dimension $n\geq 3$. As a result, the bootstrap assumption (\ref{bootstrap_assumption_energy_estimate}) is improved provided $\epsilon>0$ is sufficiently small so that $$\mathcal{E}_N[f(t)]\leq \frac{3}{2}\epsilon,$$ where we have used the smallness assumption $\mathcal{E}_N[f(0)]\leq \epsilon$ on the initial distribution function. This concludes the proof of the global energy estimate (i). Finally, note that the decay estimates in space and time for the induced spatial density in (ii) follow from applying the global energy estimate (i) combined with Proposition \ref{proposition_weighted_sobolev_absolute_values}, and Proposition \ref{proposition_improved_decay_spatial_density}.

\section{The two-dimensional case}\label{section_proof_theorem_dimension_two}

In this section, we study the evolution in time of sufficiently regular small data solutions $f$ for the Vlasov--Poisson system with the trapping potential $\frac{-|x|^2}{2}$ in dimension two. 

\subsection{The modified vector fields}\label{subsection_modified_vector_fields}

We recall the class $\lambda$ of commuting vector fields given by $$\lambda=\Big\{ U_i, S_i, L, R_{ij}  \Big\},$$ where $i,j\in \{1,2\}$. Let $(Z^i)_i$ be an arbitrary ordering of the microscopic vector fields in $\lambda$. For each vector field $Z^i\in \lambda$, we compute $$[\T_{\phi},Z^i]=\mu \sum_{k=1}^2 \partial_{x^k}(Z^i\phi+c_i\phi)\partial_{v^k},$$ where $c_i=-2$ if $Z^i=L$, otherwise, $c_i=0$. This commutator can be written in terms of the vector fields in $\lambda$ by using the identity 
\begin{equation}\label{decomposition_velocity_vector_into_stables}
\partial_{v^k}=\dfrac{1}{2e^t}e^t(\partial_{x^k}+\partial_{v^k})-\dfrac{e^t}{2}e^{-t}(\partial_{x^k}-\partial_{v^k}).
\end{equation}
Using this decomposition, we gain an exponentially growing factor, which does not allow to close the energy estimate. We avoid this problem by considering a modified set of vector fields of the form $$Y^i=Z^i-\sum_{k=1}^2 \vphi_k^i(t,x,v)S_k,$$ where $\varphi_k^i$ are sufficiently regular functions that vanish at $t=0$. For every modified vector field $Y^i$, we have $$[\T_{\phi},Y^i]=\mu\sum_{k=1}^2\partial_{x^k}(Z^i\phi+c_{i}\phi)\partial_{v^k}-\sum_{k=1}^2\T_{\phi}(\vphi^i_k)S_k-\mu\sum_{k,j=1}^2\vphi^i_k\partial_{x^j} S_k\phi\partial_{v^j}.$$ Using the decomposition \eqref{decomposition_velocity_vector_into_stables}, we have
\begin{align*}
[\T_{\phi},Y^i]&=\dfrac{\mu}{2e^t}\sum_{k=1}^2\partial_{x^k}(Z^i\phi+c_{i}\phi)U_k -\dfrac{e^t\mu}{2}\sum_{k=1}^2\partial_{x^k}(Z^i\phi+c_{i}\phi)S_k -\sum_{k=1}^2\T_{\phi}(\vphi^i_k)S_k\\
&\quad-\frac{\mu}{2e^t}\sum_{k,j=1}^2\vphi^i_k\partial_{x^j} S_k\phi U_j +\frac{e^t\mu}{2}\sum_{k,j=1}^2\vphi^i_k\partial_{x^j} S_k\phi S_j.
\end{align*}
We remove the slower decaying terms by setting $\T_{\phi}(\vphi_k^i)=-\frac{\mu}{2} e^t\partial_{x^k}(Z^i\phi+c_i\phi)$.

\begin{definition}
Let $\{Z^i\}_i$ be an ordering of $\lambda$. The modified vector fields $Y^i$ are defined as $$Y^i:=Z^i-\sum_{k=1}^n \vphi^i_k(t,x,v) S_k,\qquad S_k:=e^{-t}(\partial_{x^k}-\partial_{v^k}),$$ where 
\begin{enumerate}
\item $\vphi_k^i\equiv 0$ if $Z^i$ is a stable vector field, i.e. $Z^i=S_k$.
\item If $Z^i$ is not a stable vector field, then, $\vphi_k^i(t,x,v)$ is determined by $$\T_{\phi}(\vphi_k^i)=-\dfrac{\mu}{2} e^t\partial_{x^k}(Z^i\phi+c_i\phi),\qquad \vphi_k^i(0,x,v)=0,$$ where $c_i=-2$ if $Z^i=L$, otherwise, we set $c_i=0$.
\end{enumerate}
The set of modified vector fields is denoted by $\lambda_m$.
\end{definition}

Throughout the paper, we denote by $Y$ a generic modified vector field in $\lambda_m$. We use a multi-index notation for the microscopic differential operators of order $|\alpha|$ given by the composition $$Y^{\alpha}=Y^{\alpha_1}Y^{\alpha_2}\dots Y^{\alpha_n},$$ for every multi-index $\alpha$. We denote by $\lambda_m^{|\alpha|}$ the family of microscopic differential operators obtained as a composition of $|\alpha|$ vector fields in $\lambda_m$. We denote by $\M$ the set of all functions $\{\vphi_k^i\}$. We also denote by $\vphi$ a generic function in $\M$. 

\begin{definition}
We say that $P(\vphi)$ is a multilinear form of degree $d$ and signature less than $k$ if $P(\vphi)$ is of the form $$P(\vphi)=\sum\limits_{\substack{ |\alpha_1|+\dots+ |\alpha_d|\leq k\\ (\vphi_1, \dots, \vphi_d)\in\M^d }}C_{\bar{\alpha},\bar{\vphi}}\prod\limits_{\substack{j=1,\dots, d }}Y^{\alpha_j}(\vphi_j),$$ where $\alpha_j$ are multi-indices, and $C_{\bar{\alpha}, \bar{\vphi}}$ are uniform constants depending on $\bar{\alpha}=(\alpha_1,\dots, \alpha_d)$ and $\bar{\vphi}=(\vphi_1,\dots, \vphi_d)$. 
\end{definition}

\subsection{Properties of the modified vector fields}

In this subsection, we study the main properties of the modified vector fields that will be used later in the main proofs.

\begin{lemma}\label{lemma_commutation_formula_modified}
For any multi-index $\alpha$, we have 
\begin{equation}\label{identity_commutation_formula_modified}
[\T_{\phi},Y^{\alpha}]=\sum_{d=0}^{|\alpha|+1} \sum_{i=1}^2 \sum_{|\beta|, |\gamma|\leq |\alpha|} P^{\alpha i}_{d\gamma\beta}(\vphi)\partial_{x^i}Z^{\gamma}(\phi)Y^{\beta},
\end{equation}
where $P^{\alpha i}_{d\gamma\beta}(\vphi)$ are multilinear forms of degree $d$ and signature less than $k$ such that $k\leq |\alpha|-1$ and $k+|\gamma|+|\beta|\leq |\alpha|+1$.
\end{lemma}

\begin{proof}
As a first step, we show the commutation formula when $|\alpha|=1$. After using the equation satisfied by the coefficients $\vphi_k^i$, we have 
\begin{align*}
[\T_{\phi},Y^i]&=\dfrac{\mu}{2e^t}\sum_{k=1}^2\partial_{x^k}(Z^i\phi+c_{i}\phi)U_k -\frac{\mu}{2e^t}\sum_{k,j=1}^2\vphi^i_k\partial_{x^j} S_k\phi U_j +\frac{e^t\mu}{2}\sum_{k,j=1}^2\vphi^i_k\partial_{x^j} S_k\phi S_j.
\end{align*}
The desired identity is obtained by rewriting the unstable vector fields as $Z^k=Y^k+\sum_{l=1}^kS_l$.

The general case is proven by induction. Assume that the commutation formula holds for some multi-index $\alpha$ and let $Y\in\lambda$ be an arbitrary modified vector field. By the identity $$[\T_{\phi},YY^{\alpha}]=[\T_{\phi},Y]Y^{\alpha}+Y[\T_{\phi},Y^{\alpha}],$$ it is enough to show that the terms in the right-hand side have the correct form. The first term $[\T_{\phi},Y]Y^{\alpha}$ is treated using the case when $|\alpha|=1$. The second term $Y[\T_{\phi},Y^{\alpha}]$ generates three different types of terms. The terms $$Y(P^{\alpha i}_{d\gamma\beta}(\vphi)) \partial_{x^i}Z^{\gamma}(\phi)Y^{\beta}$$ have the correct form, since the multilinear forms have the same degree and their signature increases by one. The terms $$P^{\alpha i}_{d\gamma\beta}(\vphi) Y(\partial_{x^i}Z^{\gamma}(\phi))Y^{\beta}$$ also have the correct form, since $Y$ is schematically of the form $Z-\vphi S$, so $$P^{\alpha i}_{d\gamma\beta}(\vphi) Y(\partial_{x^i}Z^{\gamma}(\phi))Y^{\beta}=\sum_{|\gamma'|\leq |\gamma|+1}P'^{\alpha i}_{d\gamma'\beta}(\vphi) \partial_{x^i}Z^{\gamma'}(\phi)Y^{\beta},$$ where $P'^{\alpha i}_{d\gamma'\beta}$ are multilinear forms of degree at most $d+1$ with the same signature. Finally, the last terms are of the form $P^{\alpha i}_{d\gamma\beta}(\vphi) \partial_{x^i}Z^{\gamma}(\phi)YY^{\beta}$ which also satisfy the required properties. 
\end{proof}

\begin{lemma}\label{lemma_regular_vf_into_modif_vf}
For any multi-index $\alpha$, we have $$Z^{\alpha}=\sum_{d=0}^{|\alpha|} \sum_{|\beta|\leq |\alpha|} P^{\alpha}_{d\beta}(\vphi)Y^\beta,$$ where $P^{\alpha}_{d\beta}(\vphi)$ are multilinear forms of degree $d$ and signature less than $k$ with $k\leq |\alpha|-1$ and $k+|\beta|\leq |\alpha|.$
\end{lemma}

\begin{proof}
The lemma holds for $|\alpha|=1$, since $Z^i=Y^i+\sum_{k=1}^2 \vphi_k^i S_k$ where $S_k$ is also a modified vector field. An inductive argument shows the general case.
\end{proof}

\begin{lemma}\label{lemma_spatial_density_non_modified_vs_into_modified}
For every multi-index $\alpha$, we have $$\rho(Z^{\alpha}f)=\sum_{d=0}^{|\alpha|}\sum_{|\beta|\leq |\alpha|}\rho(Q_{d\beta}^{\alpha}(\partial_x\vphi)Y^{\beta}f)+\sum_{j=1}^{|\alpha|}\sum_{d=1}^{|\alpha|+1}\sum_{|\beta|\leq |\alpha|}\dfrac{1}{e^{2jt}}\rho(P_{d\beta}^{\alpha j}(\vphi)Y^{\beta}f),$$ where $Q_{d\beta}^{\alpha}(\partial_x\vphi)$ are multilinear forms with respect to $\partial_x\vphi$ of degree $d$ and signature less than $k'$ such that $k'\leq |\alpha|-1$ and $k'+d+|\beta|\leq |\alpha|$, and  $P_{d\beta}^{\alpha j}(\vphi)$ are multilinear forms of degree $d$ and signature less than $k$ such that $k\leq |\alpha|$ and $k+|\beta|\leq |\alpha|.$
\end{lemma}

\begin{proof}
Given $Z\in \lambda$, the modified vector field corresponding to $Z$ is denoted by $Y\in \lambda_m$. We will use the schematic notations $Y=Z-\vphi S$ and $Z=Y+\vphi S$ instead of the lengthy formulae given before. We will also use the notation $e^t\partial_x+e^t\partial_v-\vphi S$ to denote a generic modified vector field. We denote a generic modified vector field by $Y'$, and a generic coefficient by the letter $\vphi'\in\M$.

Firstly, we prove the lemma in the case when $|\alpha|=1$. Given $Z\in\lambda$, we have
\begin{align*}
\int Z(g)dv&=\int (Z-\vphi S+\vphi S)(g) dv\\
&=\int Y(g) dv+\int \vphi S(g) dv\\
&=\int Y(g) dv+\int \dfrac{\vphi}{e^{2t}} (e^t(\partial_xg+\partial_vg)-\vphi S(g)+\vphi S(g)-2e^t\partial_vg) dv\\
&=\int Y(g) dv+\int \dfrac{\vphi}{e^{2t}} (Y'(g)+\vphi S(g)) dv-2\int \dfrac{\vphi}{e^{t}} \partial_vg dv.
\end{align*}
The first and second terms of the right-hand side of the last line have the correct form. For the last term, we integrate by parts in the velocity variable 
\begin{align*}
-\int \dfrac{\vphi}{e^{t}} \partial_vg dv&=\dfrac{1}{e^{t}} \int  \partial_v\vphi g dv\\
&=\dfrac{1}{e^{2t}} \int  (e^t\partial_x\vphi+e^t\partial_v\vphi-\vphi'S\vphi+\vphi'S\vphi-e^t\partial_x\vphi) g dv\\
&=\dfrac{1}{e^{2t}} \int  (Y'\vphi+\vphi'S\vphi) g dv-\dfrac{1}{e^{t}} \int \partial_x(\vphi) g dv,
\end{align*}
where now all terms are of the correct form. This proves the lemma when $|\alpha|=1$. We now assume that the lemma holds true for some $\alpha$. Let $Z$ be a non-modified vector field. Using that $Z\rho(Z^{\alpha}g)=\rho(ZZ^{\alpha}g)+c_Z\rho(Z^{\alpha}g)$, we only need to show that $Z\rho(Z^{\alpha}g)$ has the correct form. Using the induction hypothesis and writing $Y=Z-\vphi S$ to denote the associated modified vector field, we have 
\begin{align*}
Z\rho(Z^{\alpha}g)&=\sum_{d=0}^{|\alpha|}\sum_{|\beta|\leq |\alpha|}\rho\Big(Z\Big[Q_{d\beta}^{\alpha}(\partial_x\vphi)Y^{\beta}f\Big]\Big)\\
&\qquad \qquad +\sum_{j=1}^{|\alpha|}\sum_{d=1}^{|\alpha|+1}\sum_{|\beta|\leq |\alpha|}\dfrac{1}{e^{2jt}}\rho\Big(Z\Big[P_{d\beta}^{\alpha j}(\vphi)Y^{\beta}f\Big]\Big)+c_Z\rho(Z^{\alpha}(g)).
\end{align*}
The last term has already the right form. Writing $Z=Y+\vphi S$ in the term $\rho\Big(Z\Big[P_{d\beta}^{\alpha j}(\vphi)Y^{\beta}f\Big]\Big)$, one easily see that they also have the desired form. For the missing term, we write
\begin{align*}
\rho\Big(Z\Big[Q_{d\beta}^{\alpha}(\partial_x\vphi)Y^{\beta}f\Big]\Big)&=\rho\Big((Y+\vphi S)\Big[Q_{d\beta}^{\alpha}(\partial_x\vphi)Y^{\beta}f\Big]\Big)\\
&=\rho\Big(Y\Big[Q_{d\beta}^{\alpha}(\partial_x\vphi)Y^{\beta}f\Big]\Big)+\rho\Big(\vphi S\Big[Q_{d\beta}^{\alpha}(\partial_x\vphi)Y^{\beta}f\Big]\Big).
\end{align*}
The first term of the right-hand side has the correct form. For the second term, we write $$\vphi S=\frac{\vphi}{e^{2t}}(e^t\partial_x+e^t\partial_v-\vphi'S+\vphi'S-2e^t\partial_v)=\frac{\vphi}{e^{2t}}Y'+\frac{\vphi\vphi'}{e^{2t}}S-2\frac{\vphi}{e^{t}}\partial_v,$$ so that $$\rho\Big(\vphi S\Big[Q_{d\beta}^{\alpha}(\partial_x\vphi)Y^{\beta}f\Big]\Big)=\dfrac{1}{e^{2t}}\rho\Big((\vphi Y'+\vphi\vphi' S)\Big[Q_{d\beta}^{\alpha}(\partial_x\vphi)Y^{\beta}f\Big]\Big)+\dfrac{2}{e^{t}}\rho\Big(\partial_v\vphi\Big[Q_{d\beta}^{\alpha}(\partial_x\vphi)Y^{\beta}f\Big]\Big),$$ where we have integrated by parts the last term. The first term on the right-hand side has the correct form. For the second term, we again write $\partial_v\vphi=\frac{1}{e^t}(e^t\partial_x+e^t\partial_v-\vphi'S+\vphi'S-e^t\partial_x)\vphi$, so that 
\begin{align*}
\dfrac{2}{e^{t}}\rho\Big(\partial_v\vphi\Big[Q_{d\beta}^{\alpha}(\partial_x\vphi)Y^{\beta}f\Big]\Big)&=\dfrac{2}{e^{2t}}\rho\Big(Y'(\vphi)\Big[Q_{d\beta}^{\alpha}(\partial_x\vphi)Y^{\beta}f\Big]\Big)+\dfrac{2}{e^{2t}}\rho\Big(\vphi'S(\vphi)\Big[Q_{d\beta}^{\alpha}(\partial_x\vphi)Y^{\beta}f\Big]\Big)\\
&\qquad -\dfrac{2}{e^{t}}\rho\Big(\partial_x\vphi\Big[Q_{d\beta}^{\alpha}(\partial_x\vphi)Y^{\beta}f\Big]\Big),
\end{align*}
where all terms now have the correct form. 
\end{proof}

\begin{lemma}\label{lemma_identity_modified_derivative_grav_field_into_regular_fields}
We have $$Y^{\alpha}\nabla \phi=Z^{\alpha}\nabla\phi+\dfrac{1}{e^{2t}}\sum_{d=1}^{|\alpha|}\sum_{|\beta|\leq |\alpha|} P^{\alpha}_{d\beta}(\vphi)Z^{\beta}\nabla\phi,$$ where $P^{\alpha}_{d\beta}(\vphi)$ are multilinear forms of degree $d$ and signature less than $k$ such that $k\leq |\alpha|-1$ and $k+|\beta|\leq |\alpha|.$
\end{lemma}

\begin{proof}
For $|\alpha|=1$, we have $$Y\nabla \phi=(Z+\vphi S)\nabla \phi=Z\nabla \phi+\dfrac{1}{e^{2t}}\vphi (e^t\partial_x)\nabla \phi.$$ An inductive argument shows the general case.
\end{proof}

\subsection{The bootstrap assumptions}

In this section, we consider distribution functions in the energy space defined in terms of the modified vector fields. For $N\geq 7$, we set the energy $$\E^m_N[f]:=\sum_{|\alpha|\leq N}\sum_{Y^{\alpha}\in \lambda_m^{|\alpha|}}\|Y^{\alpha}f\|_{L^1_{x,v}}.$$ Let $T\geq 0$ be the largest time such that, for all $t\in[0,T]$, we have
\begin{enumerate}[label = (B\arabic*)]
\item $$\E^m_N[f(t)]\leq 2\epsilon.$$ \label{boot1}
\item For every multi-index $\alpha$ with $|\alpha|\leq N-4$ and every $Y^{\alpha}\in \lambda_m^{|\alpha|}$, we have $$|Y^{\alpha}\vphi(t,x,v)|\leq \e^{\frac{1}{2}}(1+t).$$ \label{boot2}
\item For every multi-index $\alpha$ with $|\alpha|\leq N-5$ and every $Y^{\alpha}\in \lambda_m^{|\alpha|}$, we have $$|Y^{\alpha}\nabla\vphi(t,x,v)|\leq \e^{\frac{1}{2}}.$$ \label{boot3}
\item For every multi-index $\alpha$ with $|\alpha|\leq N-3$ and every $Z^{\alpha}\in \Lambda_m^{|\alpha|}$, we have $$|\nabla_x Z^{\alpha}\phi(t,x)|\leq \dfrac{\e^{\frac{1}{2}}}{e^t}.$$ \label{boot4}
\end{enumerate}

\begin{remark}
\begin{enumerate}[label = (\roman*)]
\item The modified vector fields satisfy that $Y^i=Z^i$ at time $t=0$. For this reason, the energy norm of the initial data is equal to $$\E_N[f_0]=\sum_{|\alpha|\leq N}\sum_{Z^{\alpha}\in \lambda^{|\alpha|}}\|Z^{\alpha}f_0\|_{L^1_{x,v}}\leq \e.$$
\item The bootstrap argument set in this subsection can also be used to show global existence for the Vlasov--Poisson system with the potential $\frac{-|x|^2}{2}$ in dimension greater than two. The bootstrap assumptions can be improved in higher dimensions identically as in the two-dimensional case. In the rest of the section, we will only consider $n=2$.
\end{enumerate}
\end{remark}

\subsection{Weighted Sobolev inequality with the modified vector fields}

Using the bootstrap assumptions on $\vphi$, we prove a weighted Sobolev inequality in terms of the modified vector fields.

\begin{proposition}
For every sufficiently regular distribution function $f$, the induced spatial density satisfies $$\rho(|Y^{\alpha}f|)(t,x)\lesssim \dfrac{1}{(e^t+|x|)^2} \sum_{|\beta|\leq |\alpha|+2}\sum_{Y^{\beta}\in \lambda_m^{|\beta|}}\|Y^{\beta}f\|_{L^1_{x,v}}$$ for every $t\geq 0 $, every $x\in \R^2$, and every multi-index $|\alpha|\leq N-2$.
\end{proposition}

\begin{proof}
Similarly as in the proof of Proposition \ref{proposition_weighted_sobolev_absolute_values}, we define a real-valued function $\tilde{\psi}:B_2(0,1/2)\to \R$ given by $\tilde{\psi}(y)=\rho(|Y^{\alpha}f|)(t,x+(e^t+|x|)y)$. Using a 1D Sobolev inequality with $\delta=\frac{1}{8}$, we have $$\rho(|Y^{\alpha}f|)(t,x)\lesssim \int_{|y_1|\leq \delta^{1/2}} (|\partial_{y_1}\tilde{\psi}|+|\tilde{\psi}|)(y_1,0)dy_1,$$ where as before 
\begin{align*}
\partial_{y_1}\tilde{\psi}(y)&=(e^t+|x|)\partial_{x^1}\rho(|Y^{\alpha}f|)(t,x+(t+|x|)y)\\
&=e^t\int_v \partial_{x^1}(|Y^{\alpha}f|)(t,x+(e^t+|x|)y,v)dv +|x|\int_v \partial_{x^1}(|Y^{\alpha}f|)(t,x+(e^t+|x|)y,v)dv.
\end{align*}
Now,
\begin{align*}
e^t\int_v \partial_{x^1}(|Y^{\alpha}f|)dv&=\int_v (e^t\partial_{x^1}+e^t\partial_{v^1}-\sum_{i=1}^n \vphi_{1}^i S_i+\sum_{i=1}^n \vphi_{1}^i S_i)(|Y^{\alpha}f|)dv\\
&=\int_v Y_1(|Y^{\alpha}f|)dv+\int_v \sum_{i=1}^n \vphi_{1}^i S_i(|Y^{\alpha}f|)dv.
\end{align*}
The first term on the right-hand side is simply estimated by $$\Big|\int_v Y_1(|Y^{\alpha}f|)dv\Big|\leq \int_v |Y_1Y^{\alpha}f|dv.$$ For the second term, we again make the modified vector fields to appear 

\begin{align*}
\int_v  \vphi_{1}^i S_i(|Y^{\alpha}f|)dv&=\int_v \dfrac{\vphi_{1}^i}{e^{2t}} \Big(e^t\partial_{x^i}+e^t\partial_{v^i}-\sum_{k=1}^n \vphi_i^k S_k+\sum_{k=1}^n \vphi_i^k S_k -2e^t\partial_{v^i}\Big)(|Y^{\alpha}f|)dv\\
&=\int_v \dfrac{\vphi_{1}^i}{e^{2t}} Y_i(|Y^{\alpha}f|)dv+\int_v \dfrac{\vphi_{1}^i}{e^{2t}} \Big(\sum_{k=1}^n \vphi_i^k S_k -2e^t\partial_{v^i}\Big)(|Y^{\alpha}f|)dv.
\end{align*}
The first term on the right-hand side can then be estimated as above, using that $\frac{\vphi_{1}^i}{e^{2t}}$ is uniformly bounded from the bootstrap assumptions. For the remainder terms, we first note that in view of the bootstrap assumptions, the terms of the form $$\Big|\int_v \dfrac{\vphi_{1}^i}{e^{2t}} \vphi_i^k S_k(|Y^{\alpha}f|)dv\Big|.$$ can be estimated by $$\Big|\int_v  |S_k|Y^{\alpha}f||dv\Big|\leq \int_v  |S_kY^{\alpha}f|dv.$$ For the last term, we integrate by parts in the velocity variable 
\begin{align*}
-\int_v \dfrac{\vphi_{1}^i}{e^{t}} \partial_{v^i}(|Y^{\alpha}f|)dv&=\dfrac{1}{e^t}\int_v \partial_{v^i}\vphi_{1}^i |Y^{\alpha}f|dv\\
&=\dfrac{1}{e^{2t}} \int_v \Big(e^{t}\partial_{x^i}+e^{t}\partial_{v^i}-\sum_{k=1}^n \vphi_i^kS_k+\sum_{k=1}^n \vphi_i^kS_k-e^{t}\partial_{x^i}\Big)\vphi_{1}^i  |Y^{\alpha}f|dv\\
&=\dfrac{1}{e^{2t}} \int_v Y_i(\vphi_{1}^i)  |Y^{\alpha}f|dv+\dfrac{1}{e^{2t}}\sum_{k=1}^n \int_v \vphi_i^kS_k(\vphi_{1}^i) |Y^{\alpha}f|dv\\
&\qquad -\dfrac{1}{e^{t}} \int_v \partial_{x^i}(\vphi_{1}^i)  |Y^{\alpha}f|dv.
\end{align*}
The first term only grows like $(1+t)$ according to the bootstrap assumptions, so this growth can be absorbed thanks to the exponential factor in front. The second and third terms can also be absorbed using the exponential factor in front. Putting everything together, we obtain $$\rho(|Y^{\alpha}f|)(t,x)\lesssim \int_{|y_1|\leq \delta^{\frac{1}{2}}}\int_v (|YY^{\alpha}f|+|Y^{\alpha}f|)(t,x+(e^t+|x|)(y_1,0),v)dvdy_1.$$ The remaining of the proof follows as in the proof of Proposition \ref{proposition_weighted_sobolev_absolute_values}, repeating the previous arguments for each of the variables and applying the usual change of coordinates.
\end{proof}

\subsection{Estimates for \texorpdfstring{$\|Y^{\alpha}(\vphi)Y^{\beta}(f)\|_{L^1_{x,v}}$}{bil}}

In this subsection, we prove the core estimates to close the bootstrap argument previously set. We begin proving decay estimates for $\|\nabla_x Z^{\gamma}(\phi)Y^{\alpha}f\|_{L^1_{x,v}}$.

\begin{lemma}\label{lemma_modified_weighted_Sobolev_Green_function}
For every multi-indices $\gamma$ and $\alpha$, with $|\gamma|\leq N$ and $|\alpha|\leq N-2$, we have $$\|\nabla_x Z^{\gamma}(\phi)Y^{\alpha}f\|_{L^1_{x,v}}\lesssim \dfrac{\e}{e^t}\sum_{|\beta|\leq |\gamma|}\|\rho(Z^{\beta}f)\|_{L^1_x}.$$ 
\end{lemma}

\begin{proof}
The proof is exactly the same as in the higher dimensional case, where we have used the representation formula for $\nabla_x Z^{\gamma}\phi$. The argument uses the weighted Sobolev inequality with modified vector fields. 
\end{proof}

\begin{lemma}\label{lem_main_estimate_2d}
For every sufficiently small $\sigma>0$, there exist constants $C_{\sigma}$ and $\e_{\sigma}$ such that if $\e\leq \e_{\sigma}$, then, for all multi-indices $\alpha$ and $\beta$, with $|\alpha|\leq N-1$, $|\beta|\leq N$, and $|\alpha|+|\beta|\leq N+1$, we have $$\|Y^{\alpha}(\vphi)Y^{\beta}(f)\|_{L^1_{x,v}}\leq C_{\sigma}e^{t\sigma}\e.$$ Moreover, for all multi-indices $\alpha$ and $\beta$, with $|\alpha|\leq N-2$, $|\beta|\leq N$, and $|\alpha|+|\beta|\leq N$, and all $1\leq i\leq 2$, we have $$\|Y^{\alpha}(\partial_{x^i}\vphi)Y^{\beta}(f)\|_{L^1_{x,v}}\leq C_{\sigma}\e.$$ 
\end{lemma}

\begin{proof}
Let us denote 
\begin{align*}
 \F(t)&:=\sum_{|\alpha|\leq N-1}\sum\limits_{\substack{ |\beta|\leq N\\ |\alpha|+|\beta|\leq N+1 }} \|Y^{\alpha}(\vphi)(t)Y^{\beta}(f)(t)\|_{L^1_{x,v}},\\
 \G(t)&:=\sum_{|\alpha|\leq N-2}\sum_{|\alpha|+|\beta|\leq N} \sum_{i=1}^2 \|Y^{\alpha}(\partial_{x^i}\vphi)(t)Y^{\beta}(f)(t)\|_{L^1_{x,v}}.
\end{align*}
By the bootstrap assumptions, if $|\alpha|\leq N-4$, then $$\|Y^{\alpha}(\vphi)(t)Y^{\beta}(f)(t)\|_{L^1_{x,v}}\lesssim \e^{\frac{1}{2}}(1+t)\|Y^{\beta}(f)\|_{L^1_{x,v}}\lesssim e^{t\sigma_0}\e^{\frac{3}{2}},$$ where $\sigma_0>0$ is a small constant that is to be fixed later. Similarly, if $|\alpha|\leq N-5$, we have $$\|Y^{\alpha}(\partial_{x^i}\vphi)(t)Y^{\beta}(f)(t)\|_{L^1_{x,v}}\lesssim \e^{\frac{3}{2}}.$$ If $|\alpha|>N-4$, then, we have $|\beta|\leq N-3$ since $N\geq 7$. In this case, we estimate the terms $\|Y^{\alpha}(\vphi)(t)Y^{\beta}(f)(t)\|_{L^1_{x,v}}$ through the method of characteristics $$\|Y^{\alpha}(\vphi)(t)Y^{\beta}(f)(t)\|_{L^1_{x,v}}\leq \int_0^t \|\T_{\phi}(Y^{\alpha}(\vphi)Y^{\beta}(f))\|_{L^1_{x,v}}(s)ds.$$ We decompose the term $\T_{\phi}(Y^{\alpha}(\vphi)Y^{\beta}(f))$ into three different contributions defined by $$\T_{\phi}(Y^{\alpha}(\vphi)Y^{\beta}(f))=Y^{\alpha}(\vphi)\T_{\phi}(Y^{\beta}(f))+[\T_{\phi},Y^{\alpha}](\vphi)Y^{\beta}(f)+Y^{\alpha}\T_{\phi}(\vphi)Y^{\beta}(f)=:N_1+N_2+N_3.$$

\textbf{Estimate of $N_1$.} By the commutation formula (\ref{identity_commutation_formula_modified}), we have $$N_1=\sum_{d=0}^{|\beta|+1} \sum_{i=1}^2 \sum_{|\beta'|, |\gamma|\leq |\beta|} P^{\beta i}_{d\gamma\beta'}(\vphi)\partial_{x^i}Z^{\gamma}(\phi)Y^{\beta'}(f)Y^{\alpha}(\vphi).$$ For $|\beta|\leq N-3$, the signatures of the multilinear forms $P^{\beta i}_{d\gamma\beta'}(\vphi)$ are less than $N-4$. By the bootstrap assumptions, we have $$|P^{\beta i}_{d\gamma\beta'}(\vphi)|\lesssim (1+t)^N\lesssim e^{t\sigma_0},$$ $$|\partial_{x^i}Z^{\gamma}(\vphi)|\lesssim \dfrac{\e^{\frac{1}{2}}}{e^t},$$ where $\sigma_0\in (0,1)$ is a small number to be set later. As a result, we obtain that $$\|N_1\|_{L^1_{x,v}}\lesssim \dfrac{\e^{\frac{1}{2}}}{e^{t(1-\sigma_0)}}\F(t).$$
 
\textbf{Estimate of $N_2$.} By the commutation formula (\ref{identity_commutation_formula_modified}), we have $$N_2=\sum_{d=0}^{|\alpha|+1} \sum_{i=1}^2 \sum_{|\beta'|, |\gamma|\leq |\alpha|} P^{\alpha i}_{d\gamma\beta'}(\vphi)\partial_{x^i}Z^{\gamma}(\phi)Y^{\beta'}(\vphi)Y^{\beta}(f),$$ where the multilinear forms $P^{\alpha i}_{d\gamma\beta'}$ has signature less than $k\leq |\alpha|-1$ and $k+|\gamma|+|\beta'|\leq |\alpha|+1\leq N$. If $|\gamma|\leq N-3$, then $$|\partial_{x^i}Z^{\gamma}(\phi)|\lesssim \dfrac{\e^{\frac{1}{2}}}{e^t}.$$ The term $P^{\alpha i}_{d\gamma\beta'}(\vphi)Y^{\beta'}(\vphi)$ is a multi-linear form with at most one factor $Y^{\alpha'}(\vphi)$ with $N-4<|\alpha'|\leq |\alpha|$, while the remaining terms can be uniformly bounded by $(1+t)^N\lesssim e^{t\sigma_0}$. Therefore, we obtain $$\|P^{\alpha i}_{d\gamma\beta'}(\vphi)\partial_{x^i}Z^{\gamma}(\phi)Y^{\beta'}(\vphi)Y^{\beta}(f)\|_{L^1_{x,v}}\lesssim \dfrac{\e^{\frac{1}{2}}}{e^{t(1-\sigma_0)}}\F(t).$$ If $|\gamma|>N-3$, then, by the bootstrap assumptions $$|P^{\alpha i}_{d\gamma\beta'}(\vphi)Y^{\beta'}(\vphi)|\lesssim (1+t)^N.$$ By Lemma \ref{lemma_modified_weighted_Sobolev_Green_function}, we have $$\|\nabla_x Z^{\gamma}(\phi)Y^{\beta}f\|_{L^1_{x,v}}\lesssim \dfrac{\e}{e^t}\sum_{|\eta|\leq |\gamma|}\|Z^{\eta}f\|_{L^1_{x,v}},$$ since $|\gamma|\leq |\alpha|\leq N-1.$ By Lemma \ref{lemma_regular_vf_into_modif_vf}, we have $$\|Z^{\eta}f\|_{L^1_{x,v}}\leq \sum_{d'=0}^{|\eta|}\sum_{|\eta'|\leq|\eta|}\|P^{\eta}_{d'\eta'}(\vphi)Y^{\eta'}(f)\|_{L^1_{x,v}}\lesssim (1+t)^N\F(t),$$ so we obtain $$\|P^{\alpha i}_{d\gamma\beta'}(\vphi)\partial_{x^i}Z^{\gamma}(\phi)Y^{\beta'}(\vphi)Y^{\beta}(f)\|_{L^1_{x,v}}\lesssim \e\dfrac{(1+t)^{2N}}{e^t}\F(t)\lesssim \dfrac{\e}{e^{t(1-\sigma_0)}}\F(t).$$ Putting the previous estimates together, we have $$\|N_2\|_{L^1_{x,v}}\lesssim \dfrac{\e^{\frac{1}{2}}}{e^{t(1-\sigma_0)}}\F(t).$$

\textbf{Estimate of $N_3$.} Let us recall the equation that defines the modification of the vector fields give by $$\T_{\phi}(\vphi)=e^t\sum_{i=1}^2\sum_{|\eta|\leq 1}c_{Z,i}\partial_{x^i}Z^{\eta}\phi.$$ By Lemma \ref{lemma_identity_modified_derivative_grav_field_into_regular_fields}, we have $$N_3=e^t\sum_{i=1}^2\sum_{|\eta|\leq |\alpha|+1}c_{\eta,i}\partial_{x^i}Z^{\eta}(\phi)Y^{\beta}(f)+\sum_{d=1}^{|\alpha|}\sum_{i=1}^2\sum_{|\eta|\leq |\alpha|+1}P^{\alpha}_{d\eta}(\vphi)\partial_{x^i}Z^{\eta}(\phi)Y^{\beta}(f)=:I_3^A+I_3^B,$$ where $P^{\alpha}_{d\eta}(\vphi)$ are multi-linear forms of degree $d$ with signatures less than $k$ satisfying $k\leq |\alpha|\leq N-1$ and $k+|\eta|\leq |\alpha|+1.$ If $|\eta|\leq N-3$, we have $$|\partial_{x^i}Z^{\eta}(\phi)|\lesssim \dfrac{\e^{\frac{1}{2}}}{e^t},$$ $$\|P^{\alpha}_{d\eta}(\vphi)Y^{\beta}(f)\|_{L^1_{x,v}}\lesssim (1+t)^N\F(t),$$ so we have $$\|P^{\alpha}_{d\eta}(\vphi)\partial_{x^i}Z^{\eta}(\phi)Y^{\beta}(f)\|_{L^1_{x,v}}\lesssim \dfrac{\e^{\frac{1}{2}}}{e^{t(1-\sigma_0)}}\F(t).$$ If $|\eta|> N-3$, we have $$|P^{\alpha}_{d\eta}(\vphi)|\lesssim (1+t)^N,$$ so we obtain the estimate $$\|N_3\|_{L^1_{x,v}}\lesssim e^t\sum_{|\eta|\leq |\alpha|+1}\|\partial_{x^i}Z^{\eta}(\phi)Y^{\beta}(f)\|_{L^1_{x,v}}+\dfrac{\e^{\frac{1}{2}}}{e^{t(1-\sigma_0)}}\F(t).$$ By Lemma \ref{lemma_modified_weighted_Sobolev_Green_function}, we have $$\|\partial_{x^i}Z^{\eta}(\phi)Y^{\beta}(f)\|_{L^1_{x,v}}\lesssim \sum_{|\eta'|\leq |\eta|}\dfrac{\e}{e^t}\|\rho(Z^{\eta'}f)\|_{L^1_x}\lesssim \dfrac{\e^2}{e^t}+\dfrac{\e}{e^t}\sum_{1\leq |\eta'|\leq |\eta|}\|\rho(Z^{\eta'}f)\|_{L^1_x}.$$ For $|\eta'|\geq 1$, we can write $Z^{\eta'}f=Z^{\eta''}(Zf)$ where $0\leq |\eta''|=|\eta'|-1\leq N-1.$ Applying Lemma \ref{lemma_spatial_density_non_modified_vs_into_modified} to $Zf$, we have 
\begin{align*}
\rho(Z^{\eta'}f)&=\sum_{d=0}^{|\eta''|}\sum_{|\beta'|\leq |\eta''|}\rho(Q_{d\beta'}^{\eta''}(\partial_x\vphi)Y^{\beta'}(Zf))+\sum_{j=1}^{|\eta''|}\sum_{d=1}^{|\eta''|+1}\sum_{|\beta'|\leq |\eta|}\dfrac{1}{e^{2jt}}\rho(P_{d\beta'}^{\eta'' j}(\vphi)Y^{\beta'}(Zf))\\
&=:P_1+P_2,
\end{align*}
where $Q_{d\beta'}^{\eta''}(\partial_x\vphi)$ are multilinear forms with respect to $\partial_x\vphi$ of degree $d$ and signature less than $k'$ such that $k'\leq |\eta''|-1\leq N-2$ and $k'+d+|\beta'|\leq |\eta''|$, and  $P_{d\beta'}^{\eta'' j}(\vphi)$ are multilinear forms of degree $d$ and signature less than $k$ such that $k\leq |\eta''|$ and $k+|\beta'|\leq |\eta''|\leq N-1.$ For the term $P_1$, we have 
\begin{align*}
\rho(Q_{d\beta'}^{\eta''}(\partial_x\vphi)Y^{\beta'}Zf)&=\rho(Q_{d\beta'}^{\eta''}(\partial_x\vphi)Y^{\beta'}(Yf+c_Y\vphi Sf))\\
&=\rho(Q_{d\beta'}^{\eta''}(\partial_x\vphi)Y^{\beta'}Yf)+\sum_{|\beta''|\leq|\beta'|}c_{Y\beta''}\rho(Q_{d\beta'}^{\eta''}(\partial_x\vphi)Y^{\beta''}(\vphi)Y^{\beta'-\beta''}Sf).
\end{align*}
Since $k'\leq N-2$ and $k'+d+|\beta'|+1\leq |\eta''|+1<N+1$, we have $$\|\rho(Q_{d\beta'}^{\eta''}(\partial_x\vphi)Y^{\beta'}Yf)\|_{L^1_x}\lesssim \G(t).$$ For the second contribution of the term $P_1$, we have either $k'+d\leq N-4$ or $|\beta''|\leq N-4$, so by the bootstrap assumptions $$\|\rho(Q_{d\beta'}^{\eta''}(\partial_x\vphi)Y^{\beta''}(\vphi)Y^{\beta'-\beta''}Sf)\|_{L^1_x}\lesssim \F(t)+(1+t)\G(t).$$ Therefore, the term $P_1$ satisfies $$\|P_1\|_{L^1_x}\lesssim \F(t)+e^{t\sigma_0}\G(t).$$ Using the identity $Z=Y+\vphi S$, the term $P_2$ can be estimated as $$\|P_2\|_{L^1_x}\lesssim \dfrac{(1+t)^N}{e^t}\F(t).$$ Putting the previous bounds together, we obtain $$\|N_3\|_{L^1_{x,v}}\lesssim  \e^2+\e\F(t)+\e e^{t\sigma_0}\G(t)+\dfrac{\e^{\frac{1}{2}}}{e^{t(1-\sigma_0)}}\F(t).$$ In the case when $Y^{\alpha}=Y^{\alpha'}\partial_{x^l}$, then, the term $N_3$ is given by $$N_3=e^t\sum_{i=1}^2\sum_{|\eta|\leq |\alpha|}c_{\eta,i,l}\partial_{x^i}\partial_{x^l}Z^{\eta}(\phi)Y^{\beta}(f)+\sum_{d=1}^{|\alpha|-1}\sum_{i=1}^2\sum_{|\eta|\leq |\alpha|}P^{\alpha il}_{d\eta}(\vphi)\partial_{x^i}\partial_{x^l}Z^{\eta}(\phi)Y^{\beta}(f).$$ Using the vector field $e^t\partial_{x^l}\in \Lambda$, the estimate of $N_3$ is improved by $$\|N_3\|_{L^1_{x,v}}\lesssim \sum_{|\eta|\leq |\alpha|+1}\|\partial_{x^i}Z^{\eta}(\phi)Y^{\beta}(f)\|_{L^1_{x,v}}+\dfrac{\e^{\frac{1}{2}}}{e^{t(1-\sigma_0)}}\F(t).$$ As a result, we obtain the improved estimate $$\|N_3\|_{L^1_{x,v}}\lesssim  \dfrac{\e^2}{e^t}+\dfrac{\e}{e^{t(1-\sigma_0)}}\G(t)+\dfrac{\e^{\frac{1}{2}}}{e^{t(1-\sigma_0)}}\F(t).$$ Summarizing, for every $|\alpha|>N-4$, we have $$\|\T_{\phi}(Y^{\alpha}(\vphi)Y^{\beta}f)\|_{L^1_{x,v}}\lesssim \e^2+\e\F(t)+\e e^{t\sigma_0}\G(t)+\dfrac{\e^{\frac{1}{2}}}{e^{t(1-\sigma_0)}}\F(t).$$ And for every $|\alpha|>N-5$, we have $$\|\T_{\phi}(Y^{\alpha}(\partial_x\vphi)Y^{\beta}f)\|_{L^1_{x,v}}\lesssim \dfrac{\e^2}{e^t}+\dfrac{\e}{e^{t(1-\sigma_0)}}\G(t)+\dfrac{\e^{\frac{1}{2}}}{e^{t(1-\sigma_0)}}\F(t).$$ Thus, by the method of characteristics we obtain 
\begin{align*}
\|Y^{\alpha}(\vphi)(t)Y^{\beta}(f)(t)\|_{L^1_{x,v}}&\leq \int_0^t \|\T_{\phi}(Y^{\alpha}(\vphi)Y^{\beta}(f))\|_{L^1_{x,v}}(s)ds\\
&\lesssim \e^2 t+\e^{\frac{1}{2}}\int_0^t \F(s)ds+\e \int_0^t e^{s\sigma_0}\G(s)ds,
\end{align*}
and
\begin{align*}
\|Y^{\alpha}(\partial_{x}\vphi)(t)Y^{\beta}(f)(t)\|_{L^1_{x,v}}&\leq \int_0^t \|\T_{\phi}(Y^{\alpha}(\partial_{x}\vphi)Y^{\beta}(f))\|_{L^1_{x,v}}(s)ds\\
&\lesssim \e^2+\e\int_0^t \dfrac{1}{e^{s(1-\sigma_0)}}\G(s)ds+\e^{\frac{1}{2}}\int_0^t\dfrac{1}{e^{s(1-\sigma_0)}}\F(s) ds.
\end{align*}
Therefore, we have
\begin{align*}
\F(t)&\lesssim \e^{\frac{3}{2}} e^{t\sigma_0}+\e^{\frac{1}{2}}\int_0^t \F(s)ds+\e \int_0^t e^{s\sigma_0}\G(s)ds,
 \\
\G(t)&\lesssim \e+\e\int_0^t \dfrac{1}{e^{s(1-\sigma_0)}}\G(s)ds+\e^{\frac{1}{2}}\int_0^t\dfrac{1}{e^{s(1-\sigma_0)}}\F(s) ds.
\end{align*}
Applying Gronwall's lemma to the estimate for $\F(t)$, we have $$\F(t)\lesssim\Big(\e^{\frac{3}{2}} e^{t\sigma_0}+\e \int_0^t e^{s\sigma_0}\G(s)ds\Big)e^{tC\e^{\frac{1}{2}}}.$$ Applying this estimate to the bound of $\G(t)$, we have 
\begin{align*}
\G(t)&\lesssim \e+\e\int_0^t \dfrac{1}{e^{s(1-\sigma_0)}}\G(s)ds+\e^{2}\int_0^t\dfrac{1}{e^{s(1-2\sigma_0-C\e^{\frac{1}{2}})}} ds\\
&\qquad \qquad  \qquad +\e^{\frac{3}{2}}\int_0^t\dfrac{1}{e^{s(1-\sigma_0-C\e^{\frac{1}{2}})}}\int_0^s e^{\tau\sigma_0}\G(\tau)d\tau ds,
\end{align*}
where the last term satisfies 
\begin{align*}
\e^{\frac{3}{2}}\int_0^t\dfrac{1}{e^{s(1-\sigma_0-C\e^{\frac{1}{2}})}}\int_0^s e^{\tau\sigma_0}\G(\tau)d\tau ds&=\e^{\frac{3}{2}}\int_0^t e^{\tau\sigma_0}\G(\tau) \int_{\tau}^t \dfrac{ds}{e^{s(1-\sigma_0-C\e^{\frac{1}{2}})}}d\tau.\\
&\leq \dfrac{\e^{\frac{3}{2}}}{1-\sigma_0-C\e^{\frac{1}{2}}}\int_0^t  \dfrac{\G(\tau)}{e^{\tau(1-2\sigma_0-C\e^{\frac{1}{2}})}}d\tau.
\end{align*}
Choosing $\sigma_0$ and $\e_{\sigma}$ such that $2\sigma_0+C\e_{\sigma}^{\frac{1}{2}}\leq \min \{\frac{1}{2},\sigma\}$, we have $$\G(t)\lesssim \e,\qquad \F(t)\lesssim \e e^{t\sigma}.$$
\end{proof}

\subsection{Improving the bootstrap assumptions}

In this subsection, we improve the bootstrap assumptions \ref{boot1}-\ref{boot4} by applying the estimates for the terms $\|Y^{\alpha}(\vphi)Y^{\beta}(f)\|_{L^1_{x,v}}$.

\begin{lemma}
Let $f_0$ be an initial distribution function satisfying $\E^m_N[f_0]\leq \e$. If $\e>0$ is sufficiently small, then, for all $t\in [0,T]$, we have $$\E^m_N[f(t)]\leq \dfrac{3}{2}\e.$$ 
\end{lemma}

\begin{proof}
By Lemma \ref{lemma_commutation_formula_modified}, for every multi-index $|\alpha|\leq N$, we have $$[\T_{\phi},Y^{\alpha}]f=\sum_{d=0}^{|\alpha|+1} \sum_{i=1}^2 \sum_{|\beta|, |\gamma|\leq |\alpha|} P^{\alpha i}_{d\gamma\beta}(\vphi)\partial_{x^i}Z^{\gamma}(\phi)Y^{\beta}f,$$ where $P^{\alpha i}_{d\gamma\beta}(\vphi)$ are multilinear forms of degree $d$ and signature less than $k$ such that $k\leq |\alpha|-1$ and $k+|\gamma|+|\beta|\leq |\alpha|+1$. When $|\gamma|\leq N-3$, we have $$|\partial_{x^i}Z^{\gamma}(\phi)|\leq \frac{\e^{\frac{1}{2}}}{e^t},$$ by the bootstrap assumptions. By Lemma \ref{lem_main_estimate_2d}, we have $$\|P^{\alpha,i}_{d\gamma\beta}(\vphi)Y^{\beta}(f)\|_{L^1_{x,v}} \lesssim (1+t)^{N+1}e^{t\sigma}\e,$$ since $k+|\beta|\leq N+1$ and $k\leq N-1$. By taking $\sigma>0$ small enough, we have $$\|P^{\alpha,i}_{d\gamma\beta}(\vphi)\partial_{x^i}Z^{\gamma}(\phi)Y^{\beta}(f)\|\lesssim \dfrac{\e^{\frac{3}{2}}}{e^{t\sigma'}},$$ for some $\sigma'>0$ If $|\gamma|>N-3$, we have $$|P^{\alpha,i}_{d\gamma\beta}(\vphi)|\lesssim (1+t)^{N+1},$$ since $k$, $|\beta|\leq N-4$ due to $N\geq 7$. By Lemma \ref{lemma_modified_weighted_Sobolev_Green_function}, we have $$\|\partial_{x^i}Z^{\gamma}(\phi)Y^{\beta}(f)\|_{L^1_{x,v}}\lesssim \frac{\e}{e^t}\sum_{|\eta|\leq |\gamma|} \|Z^{\eta}(f)\|_{L^1_{x,v}}.$$ By Lemma \ref{lemma_regular_vf_into_modif_vf}, we have $$Z^{\eta}(f)=\sum_{d=0}^{|\eta|} \sum_{|\eta'|\leq |\eta|} P^{\eta}_{d\eta'}(\vphi)Y^{\eta'},$$ where $P^{\eta}_{d\eta'}(\vphi)$ are multilinear forms of degree $d$ and signature less than $k$ with $k\leq |\eta|-1\leq N-1$ and $k+|\eta'|\leq |\eta|\leq N.$ By Lemma \ref{lem_main_estimate_2d}, we have $$\|Z^{\eta}(f)\|_{L^1_{x,v}}\lesssim (1+t)^N e^{t\sigma}\e,$$ which implies the existence of $\sigma'>0$ such that $$\|P^{\alpha,i}_{d\gamma\beta}(\vphi)\partial_{x^i}Z^{\gamma}(\phi)Y^{\beta}(f)\|\lesssim \dfrac{\e^{2}}{e^{t\sigma'}}.$$ Thus, there exists $\sigma'>0$ such that $$\|\T_{\phi}Y^{\alpha}(f)\|_{L^1_{x,v}}\lesssim \frac{\e^{\frac{3}{2}}}{e^{t\sigma'}}.$$ As a result, we obtain $$\E^m_N[f(t)]\leq \E^m_N[f_0]+\sum_{|\alpha|\leq N}\int_0^t \|\T_{\phi}Y^{\alpha}(f)\|_{L^1_{x,v}}\leq \e+C\e^{\frac{3}{2}}\int_0^{\infty}\frac{ds}{e^{s\sigma'}}\leq \frac{3}{2}\e,$$ when $\e>0$ is small enough.
\end{proof}

\begin{lemma}
For every multi-index $|\alpha|\leq N-3$, we have $$|\nabla_x Z^{\alpha}\phi(t,x)|\leq \dfrac{\e}{e^t}.$$
\end{lemma}

\begin{proof}
The proof follows the same strategy than the proof of Lemma \ref{lemma_decay_bilinear_terms_commuted_vlasov}. Using the Green function for the Poisson equation, we estimate the gradient $\nabla_xZ^{\gamma}\phi$ by $$|\nabla_xZ^{\gamma}\phi|(t,x)\lesssim \sum_{|\gamma'|\leq |\gamma|}\int_{\R^n} \dfrac{1}{|y|^{n-1}}\rho(|Z^{\gamma'}f|)(x-y)dy.$$ By Lemma \ref{lemma_spatial_density_non_modified_vs_into_modified}, we have $$\rho(Z^{\alpha}f)=\sum_{d=0}^{|\alpha|}\sum_{|\beta|\leq |\alpha|}\rho(Q_{d\beta}^{\alpha}(\partial_x\vphi)Y^{\beta}f)+\sum_{j=1}^{|\alpha|}\sum_{d=1}^{|\alpha|+1}\sum_{|\beta|\leq |\alpha|}\dfrac{1}{e^{2jt}}\rho(P_{d\beta}^{\alpha j}(\vphi)Y^{\beta}f),$$ where $Q_{d\beta}^{\alpha}(\partial_x\vphi)$ are multilinear forms with respect to $\partial_x\vphi$ of degree $d$ and signature less than $k'$ such that $k'\leq |\alpha|-1\leq N-4$ and $k'+d+|\beta|\leq |\alpha|\leq N-3$, and  $P_{d\beta}^{\alpha j}(\vphi)$ are multilinear forms of degree $d$ and signature less than $k$ such that $k\leq |\alpha|\leq N-3$ and $k+|\beta|\leq |\alpha|\leq N-3.$ Applying the weighted Sobolev inequality to every term in the above equation, we have
\begin{align*}
|\rho(Q_{d\beta}^{\alpha}(\partial_x\vphi)Y^{\beta}f)(x-y)|&\lesssim \dfrac{1}{(e^t+|x-y|)^2}\sum_{|\eta|\leq 2}\|Y^{\eta}[Q_{d\beta}^{\alpha}(\partial_x\vphi)Y^{\beta}f]\|_{L^1_{x,v}},\\
|\rho(P_{d\beta}^{\alpha j}(\vphi)Y^{\beta}(f))(x-y)|&\lesssim \dfrac{1}{(e^t+|x-y|)^2}\sum_{|\eta|\leq 2}\|Y^{\eta}[P_{d\beta}^{\alpha j}(\vphi)Y^{\beta}f]\|_{L^1_{x,v}}.
\end{align*}
Since $N\geq 7$, there is at most one term $Y^{\eta'}(\vphi)$ with $|\eta'|>N-4$. By the bootstrap assumption and Lemma \ref{lem_main_estimate_2d}, we obtain $$|\rho(Z^{\alpha}f)(x-y)|\lesssim \dfrac{\e}{(e^t+|x-y|)^2}+\dfrac{\e(1+t)^N}{(e^t+|x-y|)^2 e^{t}}\lesssim \dfrac{\e}{(e^t+|x-y|)^2}.$$ By Lemma \ref{lemma_uniform_integral_bound_kernel_convolution_duan}, we have $$|\nabla_xZ^{\gamma}\phi|(t,x)\lesssim \frac{\e}{e^t}.$$
\end{proof}

\begin{lemma}
For every multi-index $\alpha$ with $|\alpha|\leq N-4$, we have $$|Y^{\alpha}\vphi(t,x,v)|\leq \e(1+t).$$ Moreover, for every multi-index $\alpha$ with $|\alpha|\leq N-5$, we have $$|Y^{\alpha}\nabla\vphi(t,x,v)|\leq \e.$$
\end{lemma}

\begin{proof}
Integrating along the characteristics, we have $$|Y^{\alpha}\vphi(t,x,v)|\leq \int_0^t \|\T_{\phi}Y^{\alpha}(\vphi)(s)\|_{L^{\infty}_{x,v}}ds.$$ We estimate the two terms of the decomposition $$\T_{\phi}Y^{\alpha}(\vphi)=Y^{\alpha}\T_{\phi}(\vphi)+[\T_{\phi},Y^{\alpha}](\vphi).$$ Using the equation that defines the coefficient $\vphi$, we have $$Y^{\alpha}\T_{\phi}(\vphi)=e^t\sum_{|\eta|\leq |\alpha|+1}c_{\eta,i}\partial_{x^i}Z^{\eta}(\phi)+\sum_{d=1}^{|\alpha|}\sum_{|\eta|\leq |\alpha|+1}P^{\alpha}_{d\eta}(\vphi)\partial_{x^i}Z^{\eta}(\phi),$$ where $P^{\alpha}_{d\eta}(\vphi)$ are multi-linear forms of degree $d$ with signatures less than $k$ such that $k\leq |\alpha|\leq N-4$ and $k+|\eta|\leq |\alpha|+1\leq N-3$. By the bootstrap assumptions and the improved estimates for $\partial_{x^i}Z^{\eta}(\phi)$, we have $$|Y^{\alpha}\T_{\phi}(\vphi)|(t)\lesssim \e+\dfrac{(1+t)^{N+1}}{e^t}\lesssim \e.$$ The commutator $[\T_{\phi},Y^{\alpha}](\vphi)$ is treated using Lemma \ref{lemma_commutation_formula_modified} from where $$[\T_{\phi},Y^{\alpha}]=\sum_{d=0}^{|\alpha|+1} \sum_{i=1}^2 \sum_{|\beta|, |\gamma|\leq |\alpha|} P^{\alpha i}_{d\gamma\beta}(\vphi)\partial_{x^i}Z^{\gamma}(\phi)Y^{\beta},$$ where then multilinear form $P^{\alpha i}_{d\gamma\beta}(\vphi)$ has degree $d$ and signature less than $k$ with $k\leq |\alpha|-1\leq N-5$ and $k+|\gamma|+|\beta|\leq |\alpha|+1\leq N-3$. By the bootstrap assumptions, we have $$|[\T_{\phi},Y^{\alpha}](\vphi)|\lesssim\e \dfrac{1+(1+t)^{N+1}}{e^t}\lesssim \e.$$ Putting the previous estimates together, we have $$|Y^{\alpha}\vphi(t,x,v)|
\lesssim \int_0^t \e ds\lesssim \e(1+t).$$ Replacing the differential operator $Y^{\alpha}$ by $Y^{\alpha}\partial_x$ in the previous estimates, the term $\partial_{x^i}Z^{\eta}(\phi)$ is replaced by $\partial_{x^i}\partial_{x^j}Z^{\eta}(\phi)$ which provides additional decay since $e^{-t}\partial_{x^i}(e^t\partial_{x^j})Z^{\eta}(\phi)$. As a result, we obtain $$|Y^{\alpha}\nabla_x\vphi (t,x,v)|\lesssim \int_0^t \dfrac{\e}{e^s}ds\lesssim \e.$$
\end{proof}

In summary, we have improved the bootstrap assumptions \ref{boot1}-\ref{boot4}, and therefore the proof of Theorem \ref{theorem_stability_vacuum_external_potential_2D} is completed.

\section{The trapped set of the characteristic flow}\label{section_proof_trapped_set_characteristic_flow}

In this section, we study the trapped set $\Gamma_+$ of the characteristic flow induced by the small data solutions of \eqref{vlasov_poisson_unstable_trapping_potential_paper} that we studied in the previous sections. We give an explicit characterization of $\Gamma_+$, which coincides with the stable manifold at the origin $W^s(0,0)$.

\subsection{Properties of the trapped set} 

Let $f$ be a small data solution to the Vlasov--Poisson system with the potential $\frac{-|x|^2}{2}$, according to the assumptions in Theorem \ref{theorem_stability_vacuum_external_potential} or Theorem \ref{theorem_stability_vacuum_external_potential_2D}. Let us describe the trapped set of the particle system determined by the characteristic flow
\begin{equation}\label{characteristics_NL_system_section_trap}
\frac{d}{dt}X(t,x,v)=V(t,x,v),\qquad \frac{d}{dt}V(t,x,v)=X(t,x,v)-\mu\nabla_x\phi (t,X(t,x,v)).
\end{equation}

We have shown that for every small data solution of the system, the force field $\nabla_x\phi$ decays exponentially in time. In particular, the origin $\{x=0,v=0\}$ is formally a fixed point of \eqref{characteristics_NL_system_section_trap} when $t\to \infty$. We define the set $$W^s(0,0):=\Big\{(x,v)\in \R^n_x\times\R^n_v: (X(t,x,v),V(t,x,v))\to (0,0) \text{   as   } t\to \infty\Big\}.$$

\begin{proposition}\label{prop_stable_mfld_exist_and_propert}
The set $W^s(0,0)$ is an $n$-dimensional invariant manifold of class $C^{N-n-1}$. Moreover, the set $W^s(0,0)$ is characterized as 
\begin{equation}\label{characterization_stable_manifold}
W^s(0,0)=\Big\{(x,v): x+v=\int_0^{\infty}\frac{1}{e^{t'}}\mu\nabla_x\phi(t',X(t',x,v))dt' \Big\}.
\end{equation}
We call $W^s(0,0)$ the \emph{stable manifold of the origin}.
\end{proposition}

\begin{proof}
\textbf{Characterization of $W^s(0,0)$.} Integrating the characteristic flow \eqref{characteristics_NL_system_section_trap}, we have 
\begin{align}
X(t,x,v)+V(t,x,v)+e^t\int_0^t \dfrac{1}{e^{t'}}\mu\nabla_x\phi(t',X(t',x,v)) dt'&=e^t(x+v),\label{formula_charact_flow_x_plus_v}\\
X(t,x,v)-V(t,x,v)-e^{-t}\int_0^t e^{t'}\mu\nabla_x\phi (t',X(t',x,v)) dt'&=e^{-t}(x-v)\label{formula_charact_flow_x_minus_v}.
\end{align}
Thus, the characteristic flow $(X(t,x,v),V(t,x,v))$ satisfies 
\begin{align}
X(t,x,v)&=\frac{e^t}{2}\Big(x+v-\int_0^t \dfrac{1}{e^{t'}}\mu\nabla_x\phi(t',X(t',x,v)) dt'\Big)\label{characteristic_flow_x}\\
&\qquad \qquad \qquad+\frac{1}{2e^{t}}\Big(x-v+\int_0^t e^{t'}\mu\nabla_x\phi (t',X(t',x,v)) dt'\Big)\nonumber,\\
V(t,x,v)&=\frac{e^t}{2}\Big(x+v-\int_0^t \dfrac{1}{e^{t'}}\mu\nabla_x\phi(t',X(t',x,v)) dt'\Big) \label{characteristic_flow_v}\\
&\qquad \qquad \qquad -\frac{1}{2e^{t}}\Big(x-v+\frac{1}{2e^{t}}\int_0^t e^{t'}\mu\nabla_x\phi (t',X(t',x,v)) dt'\Big). \nonumber
\end{align}
If the dimension $n\geq 2$, then for every $t_1\geq t_2$, we have 
\begin{align*}
\Big|\int_0^{t_1} \frac{1}{e^{t'}}\mu\nabla_x\phi (t',X(t')) dt'-\int_0^{t_2} \frac{1}{e^{t'}}\mu\nabla_x\phi (t',X(t')) dt'\Big|&\lesssim \Big|\int_{t_2}^{t_1} \frac{1}{e^{t'}}\nabla_x\phi (t',X(t')) dt'\Big|\\
&\lesssim \e^{\frac{1}{2}}\int_{t_2}^{t_1} \frac{dt'}{e^{2t}}\lesssim \frac{\e^{\frac{1}{2}}}{e^{2t_2}}.
\end{align*}
Thus, the limit 
\begin{equation}\label{limits_well_defined_stable_mfld_proof}
\int_0^{\infty} \frac{1}{e^{t'}}\mu\nabla_x\phi (t',X(t')) dt',
\end{equation}
is a well-defined real value such that 
\begin{equation}\label{smallness_limits_well_defined_stable_mfld}
\Big| \int_0^{\infty} \frac{1}{e^{t'}}\mu\nabla_x\phi (t',X(t')) dt' \Big|\lesssim \e^{\frac{1}{2}}.
\end{equation}
Furthermore, for every $t\geq 0$ we have 
\begin{align*}
\Big|\int_0^{t} e^{t'}\mu\nabla_x\phi (t',X(t')) dt'\Big|&\lesssim \e^{\frac{1}{2}}\int_{0}^{t} dt'\lesssim \e^{\frac{1}{2}}t,
\end{align*}
where we have used the decay in time of the force field $\nabla_x \phi$. By the representation formulae \eqref{characteristic_flow_x} and \eqref{characteristic_flow_v}, we have $$W^s(0,0)=\Big\{(x,v): e^t\Big(x+v-\int_0^{t}\frac{1}{e^{t'}}\mu\nabla_x\phi(t',X(t',x,v))dt'\Big)\to 0 \text{ as } t\to \infty \Big\},$$ so in particular $$W^s(0,0)\subset\Big\{(x,v): x+v=\int_0^{\infty}\frac{1}{e^{t'}}\mu\nabla_x\phi(t',X(t',x,v))dt' \Big\}.$$ Furthermore, if $x+v=\int_0^{\infty}e^{-t'}\mu\nabla_x\phi(t',X(t',x,v))dt'$, then 
\begin{align*}
e^t\Big| x+v-\int_0^{t}\frac{1}{e^{t'}}\mu\nabla_x\phi(t',X(t',x,v))dt'\Big|&=e^t\Big|\int_t^{\infty}\frac{1}{e^{t'}}\nabla_x\phi(t',X(t',x,v))dt'\Big|\\
&\lesssim \e^{\frac{1}{2}} e^t\int_t^{\infty}\frac{dt'}{e^{2t'}}\lesssim \frac{\e^{\frac{1}{2}}}{e^{t}}. 
\end{align*}
Hence, we have $$\Big\{(x,v): x+v=\int_0^{\infty}\frac{1}{e^{t'}}\mu\nabla_x\phi(t',X(t',x,v))dt' \Big\}\subset W^s(0,0),$$ so the equality \eqref{characterization_stable_manifold} holds. 

\textbf{Nonemptiness of $W^s(0,0)$.} By the smallness \eqref{smallness_limits_well_defined_stable_mfld} of the limits \eqref{limits_well_defined_stable_mfld_proof}, the sets $$A_i=\Big\{(x^i,v^i)\in \R_{x^i}\times\R_{v^i}: x^i+v^i>\int_0^{\infty}\frac{1}{e^{t'}}\mu\partial_{x^i}\phi(t',X(t',x,v))dt' \Big\}$$ and $$B_i=\Big\{(x^i,v^i)\in \R_{x^i}\times\R_{v^i}: x^i+v^i<\int_0^{\infty}\frac{1}{e^{t'}}\mu\partial_{x^i}\phi(t',X(t',x,v))dt' \Big\},$$ are clearly non-empty for every $i\in \{1,2,\dots, n\}$. By the intermediate value theorem, the sets $$W^s_{i}(0,0)=\Big\{(x^i,v^i)\in \R_{x^i}\times\R_{v^i}: x^i+v^i=\int_0^{\infty}\frac{1}{e^{t'}}\mu\partial_{x^i}\phi(t',X(t',x,v))dt' \Big\}$$ are non-empty for every $i\in \{1,2,\dots, n\}$.

\textbf{Invariance of $W^s(0,0)$.} If $(x,v)\in W^s(0,0)$, then, we have $$X(t,x,v)+V(t,x,v)=e^t\int_t^{\infty}\frac{1}{e^{t'}}\mu\nabla_x\phi(t',X(t',x,v))dt',$$ by using the representation formula \eqref{formula_charact_flow_x_plus_v} and the characterization \eqref{characterization_stable_manifold} of $W^s(0,0)$. By change of variables, we have 
\begin{align*}
e^t\int_t^{\infty}\frac{1}{e^{t'}}\mu\nabla_x\phi(t',X(t',x,v))dt'&=e^t\int_0^{\infty}\frac{1}{e^{t+t'}}\mu\nabla_x\phi(t+t',X(t+t',x,v))dt'\\
&=\int_0^{\infty}\frac{1}{e^{t'}}\mu\nabla_x\phi(t+t',X(t+t',x,v))dt'. 
\end{align*}
Thus $(X(t,x,v),V(t,x,v))\in W^s(0,0)$ since $$X(t,x,v)+V(t,x,v)=\int_0^{\infty}\frac{1}{e^{t'}}\mu\nabla_x\phi(t+t',X(t+t',x,v))dt'.$$ In other words, $W^s(0,0)$ is invariant.

\textbf{Manifold structure of $W^s(0,0)$.} We define the maps $\Psi:\R^n_x\times\R^n_v\to\R^n$ and $\Phi:\R^n_x\times\R^n_v\to\R^n$, given by 
\begin{align*}
\Psi(x,v):=x+v-\Phi(x,v),\qquad \Phi(x,v):=\int_0^{\infty} \frac{1}{e^{t'}}\mu\nabla_x\phi(t',X(t',x,v))dt'.
\end{align*}

We have proved in \eqref{limits_well_defined_stable_mfld_proof}-\eqref{smallness_limits_well_defined_stable_mfld} that $\Phi$ is a well-defined map such that $|\Phi(x,v)|\leq \e^{\frac{1}{2}}$. In particular, the map $\Psi$ is also well-defined. In the following, we show that $\Psi$ and $\Phi$ are maps in the class $C^{N-n-1}$. We obtain the proposition by proving $\det[\partial_{x^j}\Psi_i](x,v)\neq 0$ for every $(x,v)\in \{(x,v):\Psi(x,v)=0\}$, and then applying the implicit function theorem.

\begin{claim}\label{claim_estimate_derivative_x_stable_mfld}
For every $(t,x,v)\in [0,\infty)\times \R^n_x\times \R^n_v$ and every $i\in \{1,2,\dots, n\}$, we have 
\begin{equation}\label{estimate_derivative_x_upper_bound}
|\partial_{x^i}X(t,x,v)|\leq (1+2\e^{\frac{1}{2}})e^t,\qquad |\partial_{v^i}X(t,x,v)|\leq (1+2\e^{\frac{1}{2}})e^t.
\end{equation}
\end{claim}

\begin{proof}
By the formula \eqref{characteristic_flow_x}, the derivatives $\partial_{x^i}X(t)$, $\partial_{v^i}X(t)$ satisfy
\begin{align}
\partial_{x^i}X(t)&=\cosh t+\frac{1}{2e^t}\int_0^t e^{t'}\mu\nabla_x(\partial_{x^i}\phi)(t',X(t'))\partial_{x^i}X (t')dt'\label{identity_derivative_x_flow_map}\\
&\qquad \qquad-\frac{e^t}{2}\int_0^t \frac{1}{e^{t'}}\mu\nabla_x(\partial_{x^i}\phi)(t',X(t'))\partial_{x^i}X(t')dt',\nonumber\\
\partial_{v^i}X(t)&=\sinh t+\frac{1}{2e^t}\int_0^t e^{t'}\mu\nabla_x(\partial_{x^i}\phi)(t',X(t'))\partial_{v^i}X (t')dt'\label{identity_derivative_v_flow_map}\\
&\qquad \qquad-\frac{e^t}{2}\int_0^t \frac{1}{e^{t'}}\mu\nabla_x(\partial_{x^i}\phi)(t',X(t'))\partial_{v^i}X(t')dt',\nonumber
\end{align}
In particular, $\partial_{x^i}X(0,x,v)=1$ and $\partial_{x^i}X(0,x,v)=0$ satisfy \eqref{estimate_derivative_x_upper_bound} for every $(0,x,v)\in [0,\infty)\times\R^n_x\times\R^n_v$. The proof of the estimate for $\partial_{x^i}X$ follows by a continuity argument. Let 
\begin{equation}\label{continuity_argument_estimate_derivative_x_stable_mfld}
T:=\sup\Big\{t\geq 0: |\partial_{x^i}X(s,x,v)|\leq (1+2\e^{\frac{1}{2}})e^s \text{ for every } s\in [0,t]\Big\}.
\end{equation}
Using the bootstrap assumption in \eqref{identity_derivative_x_flow_map}, we have
\begin{align*}
|\partial_{x^i}X(t)|&\lesssim e^t+(1+2\e^{\frac{1}{2}})\frac{\e^{\frac{1}{2}}t}{2e^t}+(1+2\e^{\frac{1}{2}})\frac{\e^{\frac{1}{2}}e^t}{2}\int_0^t \frac{dt'}{e^{2t'}}\\
&\lesssim (1+\e^{\frac{1}{2}}+2\e)e^t\lesssim \Big(1+\frac{3}{2}\e^{\frac{1}{2}}\Big)e^t.
\end{align*}
Therefore, the supremum \eqref{continuity_argument_estimate_derivative_x_stable_mfld} is infinite, and we obtain the desired estimate for $\partial_{x^i}X$. The same argument proves the estimate for $\partial_{v^i}X$.
\end{proof}

By Claim \ref{claim_estimate_derivative_x_stable_mfld}, for every $t_1\geq t_2$ we have
\begin{align*}
\Big|\int_0^{t_1} \frac{1}{e^{t'}}\mu\nabla_x(\partial_{x^i}\phi)(t',X(t'))\partial_{x^i}X(t') dt'-\int_0^{t_2}& \frac{1}{e^{t'}}\mu\nabla_x(\partial_{x^i}\phi)(t',X(t'))\partial_{x^i}X(t') dt'\Big|\\
&\lesssim \Big|\int_{t_2}^{t_1} \frac{1}{e^{t'}}\nabla_x(\partial_{x^i}\phi)(t',X(t'))\partial_{x^i}X(t') dt'\Big|\\
&\lesssim \e^{\frac{1}{2}}(1+2\e^{\frac{1}{2}})\int_{t_2}^{t_1} \frac{dt'}{e^{2t}}\lesssim \e^{\frac{1}{2}}\frac{1}{e^{2t_2}}.
\end{align*}
Thus, the limit 
\begin{equation}\label{limits_well_defined_first_deriv_stable_mfld_proof}
\int_0^{\infty} \frac{1}{e^{t'}}\mu\nabla_x(\partial_{x^i}\phi)(t',X(t'))\partial_{x^i}X(t') dt',
\end{equation}
is a well-defined real value such that 
\begin{equation}\label{smallness_limits_first_deriv_well_defined_stable_mfld}
\Big| \int_0^{\infty} \frac{1}{e^{t'}}\mu\nabla_x(\partial_{x^i}\phi)(t',X(t'))\partial_{x^i}X(t') dt' \Big|\lesssim \e^{\frac{1}{2}}.
\end{equation}
Thus, the integral $$\int_0^{t} \frac{1}{e^{t'}}\mu\nabla_x(\partial_{x^i}\phi)(t',X(t',x,v))\partial_{x^i}X(t',x,v)dt'$$ converges uniformly with respect to $(x,v)\in \R^n_x\times\R^n_v$. Using the continuity of the derivative $\partial_{x^i}(\nabla_x\phi(t',X(t',x,v)))=\nabla_x(\partial_{x^i}\phi)(t',X(t',x,v))\partial_{x^i}X(t',x,v)$ for every $(t,x,v)\in [0,\infty)\times \R^n_x\times\R^n_v$, then
\begin{equation}
\partial_{x^i}\Phi(x,v)=\int_0^{\infty} \frac{1}{e^{t'}}\mu\nabla_x(\partial_{x^i}\phi)(t',X(t',x,v))\partial_{x^i}X(t',x,v)dt'
\end{equation}
is well-defined. Furthermore, the estimate $|\partial_{x^i}\Phi(x,v)|\lesssim \e^{\frac{1}{2}}$ holds. The same argument shows that $\partial_{v^i}\Phi(x,v)$ is well-defined and $|\partial_{v^i}\Phi(x,v)|\lesssim \e^{\frac{1}{2}}$. We have proved that $\Phi$ and $\Psi$ are maps of class $C^1$. Next, we proceed to show that $\Phi$ and $\Psi$ are actually maps of class $C^{N-n-1}$. 

In the following claim, we will use the multivariate Faà di Bruno formula \cite[Theorem 2.1]{CS96} to estimate the partial derivatives of $\nabla_x\phi(t,X(t))$. For this purpose, we introduce a linear order in $\N^{2n}_0$. If $\mu=(\mu_1,\dots,\mu_{2n})$ and $\nu=(\nu_1,\dots,\nu_{2n})$ belong to $\N_0^{2n}$, we write $\mu\prec \nu$ provided one of the following holds:
\begin{enumerate}[label = (\roman*)]
\item $|\mu|\leq |\nu|$.
\item $|\mu|= |\nu|$ and $\mu_1<\nu_1$.
\item $|\mu| = |\nu|$, $\mu_1=\nu_1$, \dots, $\mu_k=\nu_k$, and $\mu_{k+1}<\nu_{k+1}$ for some $1\leq k\leq 2n$.
\end{enumerate}

\begin{claim}\label{claim_estimate_higher_derivative_x_stable_mfld}
For every $(t,x,v)\in [0,\infty)\times \R^n_x\times \R^n_v$ and every $2\leq |\alpha|\leq N-n-1$, we have 
\begin{equation}\label{estimate_derivative_higher_x_upper_bound}
|\partial^{\alpha}_{x,v}X(t,x,v)|\leq (1+2\e^{\frac{1}{2}})e^t.
\end{equation}
\end{claim}

\begin{proof}
By the formula \eqref{characteristic_flow_x}, the derivative $\partial^{\alpha}_{x,v}X(t)$ satisfies
\begin{equation}\label{formula_higher_derivatives_x_characteristics}
\partial_{x,v}^{\alpha}X(t)=\frac{1}{2e^t}\int_0^t e^{t'}\mu\partial_{x,v}^{\alpha}(\nabla_x\phi(t',X(t')))dt'-\frac{e^t}{2}\int_0^t \frac{1}{e^{t'}}\mu\partial_{x,v}^{\alpha}(\nabla_x\phi(t',X(t')))dt'.
\end{equation}
In particular, $\partial_{x,v}^{\alpha}X(0,x,v)=0$ satisfies \eqref{estimate_derivative_higher_x_upper_bound} for every $(0,x,v)\in [0,\infty)\times\R^n_x\times\R^n_v$. Suppose that \eqref{estimate_derivative_higher_x_upper_bound} holds for every derivative $\partial_{x,v}^{\beta}X$ with $|\beta|< |\alpha|$. If $|\alpha|=2$ the estimate \eqref{estimate_derivative_higher_x_upper_bound} holds for every $\partial_{x,v}^{\beta}X$ with $|\beta|< 2$ by Claim \ref{claim_estimate_derivative_x_stable_mfld}. The proof of the estimate \eqref{estimate_derivative_higher_x_upper_bound} follows by a continuity argument. Let 
\begin{equation}\label{continuity_argument_estimate_higher_derivative_x_stable_mfld}
T:=\sup\Big\{t\geq 0: |\partial_{x,v}^{\alpha}X(s,x,v)|\leq (1+2\e^{\frac{1}{2}})e^s \text{ for every } s\in [0,t], \text{ and every } ||\Big\}.
\end{equation}
By the multivariate Faà di Bruno formula \cite[Theorem 2.1]{CS96}, we have 
\begin{equation}\label{formula_faa_di_bruno}
\partial_{x,v}^{\alpha}(\nabla_x\phi(t,X(t)))=\sum_{1\leq |\lambda|\leq |\alpha|}\nabla_x\partial_x^{\lambda}\phi(t,X(t))\sum_{s=1}^{|\alpha|}\sum_{p_s(\alpha,\lambda)}(\alpha !)\prod_{j=1}^s \frac{\prod_{i=1}^{2n}(\partial_{x,v}^{l_j}X^i)^{k^i_j}}{k_j! (l_j!)^{|k_j|}},
\end{equation}
where 
\begin{align*}
p_s(\alpha,\lambda)=\Big\{(k_1,\dots,k_s; l_1,\dots,l_s)\in (\N_0^{2n})^{2s}: |k_i|&>0, \quad 0\prec l_1 \prec \dots, \prec l_s,\\  
&\quad \quad\sum_{i=1}^s k_i= \lambda,  \quad \sum_{i=1}^s |k_i| l_i=\nu \Big\}.
\end{align*}
Using the bootstrap assumption to estimate the derivative $\partial_{x,v}^{\alpha}(\nabla_x\phi(t,X(t)))$, we have 
\begin{align}
|\partial_{x,v}^{\alpha}(\nabla_x\phi(t,X(t)))|&\lesssim \e^{\frac{1}{2}}\sum_{1\leq |\lambda|\leq |\alpha|}(1+2\e^{\frac{1}{2}})^{|\lambda|}e^{-t(1+|\lambda|)}e^{t\sum_{j=1}^s|k_j|} \label{estimate_derivative_force_field_along_flow}\\ &\lesssim \e^{\frac{1}{2}}e^{-t}\sum_{1\leq |\lambda|\leq |\alpha|}(1+2\e^{\frac{1}{2}})^{|\lambda|}\lesssim \e^{\frac{1}{2}}e^{-t},\nonumber
\end{align}
where we have used the decay in time of the force field. Applying \eqref{estimate_derivative_force_field_along_flow} in the representation formula \eqref{formula_higher_derivatives_x_characteristics}, we obtain 
\begin{align*}
|\partial_{x,v}^{\alpha}X(t)|&\lesssim \dfrac{t}{e^t}\e^{\frac{1}{2}}+e^t\e^{\frac{1}{2}}\lesssim \Big(1+\frac{3}{2}\e^{\frac{1}{2}}\Big)e^t.
\end{align*}
Therefore, the supremum \eqref{continuity_argument_estimate_higher_derivative_x_stable_mfld} is infinite, and we obtain the desired estimate for $\partial_{x,v}^{\alpha}X(t)$.
\end{proof}

By Claim \ref{claim_estimate_higher_derivative_x_stable_mfld}, for every $t_1\geq t_2$ we have
\begin{align*}
\Big|\int_0^{t_1} \frac{1}{e^{t'}}\mu\partial_{x,v}^{\alpha}(\nabla_x\phi(t',X(t'))) dt'-\int_0^{t_2}& \frac{1}{e^{t'}}\mu\partial_{x,v}^{\alpha}(\nabla_x\phi(t',X(t'))) dt'\Big|\\
&\lesssim \Big|\int_{t_2}^{t_1} \frac{1}{e^{t'}}\partial_{x,v}^{\alpha}(\nabla_x\phi(t',X(t'))) dt'\Big|\\
&\lesssim \e^{\frac{1}{2}}\int_{t_2}^{t_1} \frac{dt'}{e^{2t}}\lesssim \e^{\frac{1}{2}}\frac{1}{e^{2t_2}}.
\end{align*}
Thus, the limit 
\begin{equation}\label{limits_well_defined_higher_stable_mfld_proof}
\int_0^{\infty} \frac{1}{e^{t'}}\mu\partial_{x,v}^{\alpha}(\nabla_x\phi(t',X(t'))) dt',
\end{equation}
is a well-defined real value such that 
\begin{equation}\label{smallness_limits_well_defined_higher_stable_mfld}
\Big| \int_0^{\infty} \frac{1}{e^{t'}}\mu\partial_{x,v}^{\alpha}(\nabla_x\phi(t',X(t'))) dt' \Big|\lesssim \e^{\frac{1}{2}}.
\end{equation}
Thus, the integral $$\int_0^{t} \frac{1}{e^{t'}}\mu\partial_{x,v}^{\alpha}(\nabla_x\phi(t',X(t',x,v)))dt'$$ converges uniformly with respect to $(x,v)\in \R^n_x\times\R^n_v$. Using the continuity of the derivative $\partial_{x,v}^{\alpha}(\nabla_x\phi(t,X(t,x,v)))$ for every $(t,x,v)\in [0,\infty)\times \R^n_x\times\R^n_v$, then
\begin{equation}
\partial_{x,v}^{\alpha}\Phi(x,v)=\int_0^{\infty} \frac{1}{e^{t'}}\mu\partial_{x,v}^{\alpha}(\nabla_x\phi(t',X(t',x,v)))dt'
\end{equation}
is well-defined. Furthermore, the estimate $|\partial_{x,v}^{\alpha}\Phi(x,v)|\lesssim \e^{\frac{1}{2}}$ holds. We have proved that $\Phi$ and $\Psi$ are maps of class $C^{N-n-1}$. 

As a result, for every $(x,v)\in \{(x,v):\Psi(x,v)=0\}$ we have $$\det[\partial_{x^j}\Psi_i](x,v)=\det[\delta_{ij}-\partial_{x^j}\Phi_i](x,v)> 0,$$ since $|\partial_{x^i}\Phi(x,v)|\lesssim \e^{\frac{1}{2}}$. By the implicit function theorem, we conclude that $W^s(0,0)$ is an $n$-dimensional manifold of class $C^{N-n-1}$.
\end{proof}

\begin{corollary}\label{cor_characteriz_trapped_set}
The trapped set $\Gamma_+$ of the characteristic flow \eqref{characteristics_NL_system_section_trap} is equal to the stable manifold of the origin $W^s(0,0)$.
\end{corollary}

\begin{proof}
By the representation formulae \eqref{characteristic_flow_x} and \eqref{characteristic_flow_v}, if $x+v\neq \int_0^{\infty}e^{-t'}\mu\nabla_x\phi(t',X(t',x,v))dt'$, then $|X(t,x,v)|\to \infty$ and $|V(t,x,v)|\to \infty$. In other words, every $(x,v)\in \R^n_x\times\R^n_v\setminus W^s(0,0)$ escapes to infinity. In contrast, every $(x,v)\in W^s(0,0)$ is trapped by definition.
\end{proof}

\begin{proof}[Proof of Theorem \ref{thm_characterization_trapped_set}]
Apply Proposition \ref{prop_stable_mfld_exist_and_propert} and Corollary \ref{cor_characteriz_trapped_set}.
\end{proof}


\bibliographystyle{alpha}
\bibliography{Bibliography.bib} 

\end{document}